\documentclass[reqno, 12pt]{amsart}%
\usepackage{amsmath}
\usepackage{amssymb}
\usepackage{amsfonts}
\usepackage{graphicx}%
\providecommand{\U}[1]{\protect\rule{.1in}{.1in}}
\newtheorem{theorem}{Theorem}
\theoremstyle{plain}
\newtheorem{acknowledgement}{Acknowledgement}

\newtheorem{claim}{Claim}

\newtheorem{corollary}{Corollary}

\newtheorem{definition}{Definition}

\newtheorem{lemma}{Lemma}

\newtheorem{proposition}{Proposition}
\newtheorem{remark}{Remark}

\DeclareMathOperator{\Div}{div}
\numberwithin{equation}{section}
\numberwithin{theorem}{section}
\numberwithin{proposition}{section}
\numberwithin{remark}{section}
\numberwithin{definition}{section}
\numberwithin{lemma}{section}
\numberwithin{corollary}{section}
\numberwithin{example}{section}
\numberwithin{claim}{section}
\begin{document}
\title[Deterministic homogenization of elliptic equations]{Deterministic homogenization of elliptic equations with lower order terms}
\author{Renata Bunoiu}
\address{R. Bunoiu, IECL, CNRS UMR 7502, Universit\'{e} de Lorraine, 3, rue Augustin
Fresnel, 57073, Metz, France}
\email{renata.bunoiu@univ-lorraine.fr}
\urladdr{http://www.iecl.univ-lorraine.fr/\symbol{126}Renata.Bunoiu/}
\author{Giuseppe Cardone}
\address{G. Cardone, University of Sannio, Department of Engineering, Corso Garibaldi,
107, 84100 Benevento, Italy}
\email{giuseppe.cardone@unisannio.it}
\urladdr{http://www.ing.unisannio.it/cardone}
\author{Willi J\"{a}ger}
\address{W. J\"{a}ger, Interdisciplinary Center for Scientific Computing (IWR),
University of Heidelberg, Im Neuenheimer Feld 205, 69120 Heidelberg, Germany.}
\email{wjaeger@iwr.uni-heidelberg.de}
\author{Jean Louis Woukeng}
\address{J. L. Woukeng, Department of Mathematics and Computer Science, University of
Dschang, P.O. Box 67, Dschang, Cameroon}
\curraddr{J. L. Woukeng, Interdisciplinary Center for Scientific Computing (IWR),
University of Heidelberg, Im Neuenheimer Feld 205, Mathematikon 205, 69120
Heidelberg, Germany}
\email{jwoukeng@gmail.com}
\thanks{}
\date{October 30, 2019}
\subjclass[2000]{ 35B40, 46J10}
\keywords{Convergence rates, deterministic homogenization, sigma-convergence}

\begin{abstract}
For a class of linear elliptic equations of general type with rapidly
oscillating coefficients, we use the sigma-convergence method to prove the
homogenization result and a corrector-type result. In the case of asymptotic
periodic coefficients we derive the optimal convergence rates for the zero
order approximation of the solution with no smoothness on the coefficients, in
contrast to what has been done up to now in the literature. This follows as a
result of the existence of asymptotic periodic correctors for general
nonsmooth coefficients. The homogenization process is achieved through a
compactness result obtained by proving a Helmholtz-type decomposition theorem
in case of Besicovitch spaces.

\end{abstract}
\maketitle

\section{Introduction and the main results\label{sec1}}

In this work we are interested in the homogenization from both qualitative and
quantitative standpoints, of a family of general elliptic operators of the
type
\begin{equation}
\mathcal{P}_{\varepsilon}=-\operatorname{div}(A^{\varepsilon}\nabla
+V^{\varepsilon})+B^{\varepsilon}\nabla+a_{0}^{\varepsilon}+\mu\label{1.1}%
\end{equation}
where $\varepsilon>0,\mu\geq0$, $A^{\varepsilon}(x)=A(x/\varepsilon)$, (and a
similar definition for $V^{\varepsilon}$, $B^{\varepsilon}$ and $a_{0}%
^{\varepsilon}$). The coefficients $A$, $V$, $B\ $and $a_{0}$ are constrained
as follows:

\begin{itemize}
\item[(H1)] The symmetric matrix $A$ has entries in $L^{\infty}(\mathbb{R}%
^{d})$ ($d\geq2$) and there are $\alpha,\beta>0$ such that
\begin{equation}
\alpha\left\vert \xi\right\vert ^{2}\leq A(y)\xi\cdot\xi\leq\beta\left\vert
\xi\right\vert ^{2},\text{ a.e. }y\in\mathbb{R}^{d}\text{ and all }\xi
\in\mathbb{R}^{d}; \label{1.2}%
\end{equation}

\item[(H2)] $B,V\in L^{\infty}(\mathbb{R}^{d})^{d}$ and $a_{0}\in L^{\infty
}(\mathbb{R}^{d})$ with
\begin{equation}
\max\{\left\Vert B\right\Vert _{L^{\infty}(\mathbb{R}^{d})},\left\Vert
V\right\Vert _{L^{\infty}(\mathbb{R}^{d})},\left\Vert a_{0}\right\Vert
_{L^{\infty}(\mathbb{R}^{d})}\}\leq\alpha_{0} \label{1.3}%
\end{equation}
where $\alpha_{0}>0$ is given.
\end{itemize}

We further assume that all the above coefficients belong to a Besicovitch
space associated to an algebra with mean value in the sense given below in (H3):

\begin{itemize}
\item[(H3)] The functions $A=(a_{ij})_{1\leq i,j\leq d}$, $B=(b_{i})_{1\leq
i\leq d}$, $V=(v_{i})_{1\leq i\leq d}$ and $a_{0}$ are structured as follows:
$A\in B_{\mathcal{A}}^{2}(\mathbb{R}^{d})^{d\times d}$, $B,V\in B_{\mathcal{A}%
}^{2}(\mathbb{R}^{d})^{d}$ and $a_{0}\in B_{\mathcal{A}}^{2}(\mathbb{R}^{d})$.
\end{itemize}

In (H3), $\mathcal{A}$ is a given algebra with mean value on $\mathbb{R}^{d}$
(that is, a closed subalgebra of the $\mathcal{C}^{\ast}$-Banach algebra of
bounded uniformly continuous real-valued functions on $\mathbb{R}^{d}$ that
contains the constants, is translation invariant and is such that each of its
elements $u$ possesses a mean value
$M(u)=\mathchoice {{\setbox0=\hbox{$\displaystyle{\textstyle
-}{\int}$ } \vcenter{\hbox{$\textstyle -$
}}\kern-.6\wd0}}{{\setbox0=\hbox{$\textstyle{\scriptstyle -}{\int}$ }
\vcenter{\hbox{$\scriptstyle -$
}}\kern-.6\wd0}}{{\setbox0=\hbox{$\scriptstyle{\scriptscriptstyle -}{\int}$
} \vcenter{\hbox{$\scriptscriptstyle -$
}}\kern-.6\wd0}}{{\setbox0=\hbox{$\scriptscriptstyle{\scriptscriptstyle
-}{\int}$ } \vcenter{\hbox{$\scriptscriptstyle -$ }}\kern-.6\wd0}}\!\int
_{B_{R}}u(y)dy$ where we set once and for all
$\mathchoice {{\setbox0=\hbox{$\displaystyle{\textstyle
-}{\int}$ } \vcenter{\hbox{$\textstyle -$
}}\kern-.6\wd0}}{{\setbox0=\hbox{$\textstyle{\scriptstyle -}{\int}$ }
\vcenter{\hbox{$\scriptstyle -$
}}\kern-.6\wd0}}{{\setbox0=\hbox{$\scriptstyle{\scriptscriptstyle -}{\int}$
} \vcenter{\hbox{$\scriptscriptstyle -$
}}\kern-.6\wd0}}{{\setbox0=\hbox{$\scriptscriptstyle{\scriptscriptstyle
-}{\int}$ } \vcenter{\hbox{$\scriptscriptstyle -$ }}\kern-.6\wd0}}\!\int
_{K}=\frac{1}{\left\vert K\right\vert }\int_{K}$ for any measurable set
$K\subset\mathbb{R}^{d}$) and $B_{\mathcal{A}}^{2}(\mathbb{R}^{d})$ is the
generalized Besicovitch space defined as the closure of the algebra
$\mathcal{A}$ with respect to the seminorm $\left\Vert u\right\Vert
_{2}=(M(\left\vert u\right\vert ^{2}))^{\frac{1}{2}}$.

Assumption (H3) is crucial in the process of homogenization. Without such an
assumption, we can not take full advantage of the inner structure (such as the
mean value property which is fundamental) of the coefficients and proceed with
the homogenization process. In the remainder of the work, unless otherwise
stated, we will always assume implicitly that (H1), (H2) and (H3) are satisfied.

To the operator $\mathcal{P}_{\varepsilon}$ is associated the bilinear form
\begin{equation}
\mathcal{B}_{\varepsilon,\Omega}(u,v)=\int_{\Omega}(A^{\varepsilon}\nabla
u+V^{\varepsilon}u)\cdot\nabla vdx+\int_{\Omega}(B^{\varepsilon}\nabla
u+a_{0}^{\varepsilon}u+\mu u)vdx,\ \ u,v\in H^{1}(\Omega) \label{*0}%
\end{equation}
where $\Omega$ is a bounded open set in $\mathbb{R}^{d}$. Then in view of (H1)
and (H2), it holds that there exist constants $C_{i}=C_{i}(\alpha,\beta
,\alpha_{0},d)>0$ ($i=1,2$) and $\mu_{0}=\mu_{0}(\alpha,\beta,\alpha
_{0},d)\geq0$ such that for all $u,v\in H_{0}^{1}(\Omega)$,
\[
\left\vert \mathcal{B}_{\varepsilon,\Omega}(u,v)\right\vert \leq
C_{1}\left\Vert u\right\Vert _{H^{1}(\Omega)}\left\Vert v\right\Vert
_{H^{1}(\Omega)}\text{ and }C_{2}\left\Vert u\right\Vert _{H_{0}^{1}(\Omega
)}^{2}\leq\mathcal{B}_{\varepsilon,\Omega}(u,u)+\mu_{0}\left\Vert u\right\Vert
_{L^{2}(\Omega)}^{2}.
\]
It therefore follows that, under assumptions (H1) and (H2), and for each $f\in
L^{2}(\Omega)$ and $F\in L^{2}(\Omega)^{d}$, there exists (see \cite[Theorems
2 and 3, Chap. 6, Section 6.2]{Evans}) a unique weak solution $u_{\varepsilon
}\in H_{0}^{1}(\Omega)$ to the equation
\begin{equation}
\mathcal{P}_{\varepsilon}u_{\varepsilon}=f+\operatorname{div}F\text{ in
}\Omega\label{*10}%
\end{equation}
whenever $\mu\geq\mu_{0}$. The solution $u_{\varepsilon}$ satisfies the
estimate
\begin{equation}
\sup_{\varepsilon>0}\left\Vert u_{\varepsilon}\right\Vert _{H_{0}^{1}(\Omega
)}\leq C(\left\Vert f\right\Vert _{L^{2}(\Omega)}+\left\Vert F\right\Vert
_{L^{2}(\Omega)}) \label{*11}%
\end{equation}
where $C=C(\alpha,\beta,\alpha_{0},d)>0$. We will assume for the rest of the
work that $\mu\geq\mu_{0}$. The following theorem is one of the main results
of the work. It is concerned with the qualitative theory for (\ref{1.1}).

\begin{theorem}
\label{t1.1}Let $\Omega$ be a $\mathcal{C}^{0,1}$ bounded domain in
$\mathbb{R}^{d}$. Let the assumptions \emph{(H1)-(H3)} be satisfied. Let
$(u_{\varepsilon})_{\varepsilon>0}$ be the sequence of solutions to
\emph{(\ref{*10})}. There exists $u_{0}\in H_{0}^{1}(\Omega)$ such that as
$\varepsilon\rightarrow0$, one has $u_{\varepsilon}\rightarrow u_{0}$ weakly
in $H_{0}^{1}(\Omega)$ and strongly in $L^{2}(\Omega)$, and $u_{0}$ is the
unique solution to the problem
\begin{equation}
-\operatorname{div}(\widehat{A}\nabla u_{0}+\widehat{V}u_{0})+\widehat{B}%
\cdot\nabla u_{0}+\widehat{a}_{0}u_{0}+\mu u_{0}=f+\operatorname{div}F\text{
\ in }\Omega\label{1.4'}%
\end{equation}
where $\widehat{A}$, $\widehat{B}$, $\widehat{V}$ and $\widehat{a}_{0}$ are
the homogenized coefficients defined by
\begin{equation}
\widehat{A}=M\left(  A(I_{d}+\nabla_{y}\chi)\right)  \text{, }\widehat
{B}=M\left(  B(I_{d}+\nabla_{y}\chi)\right)  \text{, }\widehat{V}%
=M(A\nabla_{y}\chi_{0}+V\mathbf{)}\text{ and }\widehat{a}_{0}=M(B\cdot
\nabla_{y}\chi_{0}+a_{0}). \label{1.5'}%
\end{equation}
Here $I_{d}$ is the identity matrix of order $d$, $\chi=(\chi_{j})_{1\leq
j\leq d}\in B_{\#\mathcal{A}}^{1,2}(\mathbb{R}^{d})^{d}$ and $\chi_{0}\in
B_{\#\mathcal{A}}^{1,2}(\mathbb{R}^{d})$ are the unique solutions (up to an
additive constant) of the corrector problems
\begin{align}
-\operatorname{div}_{y}\left(  A(e_{j}+\nabla_{y}\chi_{j})\right)   &
=0\text{ in }\mathbb{R}^{d},\label{1.6'}\\
-\operatorname{div}_{y}(A\nabla_{y}\chi_{0}+V)  &  =0\text{ in }\mathbb{R}^{d}
\label{1.66'}%
\end{align}
respectively. If we let
\[
u_{1}(x,y)=\chi(y)\nabla u_{0}(x)+\chi_{0}(y)u_{0}(x)=\sum_{j=1}^{d}%
\frac{\partial u_{0}}{\partial x_{j}}(x)\chi_{j}(y)+\chi_{0}(y)u_{0}(x)
\]
and further assume that $u_{1}\in H^{1}(\Omega;\mathcal{A}^{1})$
($\mathcal{A}^{1}=\{v\in\mathcal{A}:\nabla_{y}v\in(\mathcal{A})^{d}\}$), then
as $\varepsilon\rightarrow0$,
\begin{equation}
u_{\varepsilon}-u_{0}-\varepsilon u_{1}^{\varepsilon}\rightarrow0\text{ in
}H^{1}(\Omega)\text{-strongly} \label{1.7'}%
\end{equation}
where $u_{1}^{\varepsilon}(x)=u_{1}(x,x/\varepsilon)$ for a.e. $x\in\Omega$.
\end{theorem}

In (\ref{1.4'}), it is easy to see that the coefficients $\widehat{A}$,
$\widehat{V}$, $\widehat{B}$ and $\widehat{a}_{0}$ (which are constant)
satisfy properties similar to (H1) and (H2), so that the uniqueness of the
solution to (\ref{1.4'}) is ensured by \cite[Theorem 3, Chap. 6, Section
6.2]{Evans} as well.

Although the qualitative homogenization result (\ref{1.4'}) is classically
known, it seems however useful to recall it here and provide a self-contained
proof. This is important because of the fact that in the definition of the
homogenized coefficients (\ref{1.5'}), we make use of the correctors $\chi
_{j}$ ($0\leq j\leq d$) that are here obtained as the solutions in the usual
sense of distributions, of problems (\ref{1.6'}) and (\ref{1.66'}). The
resolution of these equations in the distributional sense allows the
computation of the homogenized coefficients, and so of the solution $u_{0}$. A
numerical scheme is provided in \cite{JTW} in order to approximate the
homogenized coefficients, and numerical simulations are also provided therein.

Theorem \ref{t1.1} holds for any algebra with mean value, and is the main
qualitative result of this work. Concerning the quantitative aspect, there is
no general theory since the rates of convergence rely heavily on the intrinsic
properties of the coefficients of the operator. Depending on the existence or
not of true correctors, the convergence rates can be optimal or not. We
consider in this study a special framework that fits the setting of defect
periodic media. To be more precise, we assume that (H1) and (H2) still hold
true and we replace (H3) by (\ref{2.1})-(\ref{2.2}) below%

\begin{equation}
A=A_{0}+A_{per}\in L_{\infty,per}^{2}(\mathbb{R}^{d})^{d\times d}=(L_{0}%
^{2}(\mathbb{R}^{d})+L_{per}^{2}(Y))^{d\times d}\text{, }V\in L_{\infty
,per}^{2}(\mathbb{R}^{d})^{d} \label{2.1}%
\end{equation}%
\begin{equation}
B\in L_{\infty,per}^{2}(\mathbb{R}^{d})^{d}\text{ and }a_{0}\in L_{\infty
,per}^{2}(\mathbb{R}^{d}%
)\ \ \ \ \ \ \ \ \ \ \ \ \ \ \ \ \ \ \ \ \ \ \ \ \ \ \ \ \ \ \ \ \ \ \ \ \ \ \ \ \ \ \ \ \ \ \ \ \label{2.2'}%
\end{equation}
where the matrix $A_{per}$ is symmetric and further satisfies
\begin{equation}
\alpha\left\vert \xi\right\vert ^{2}\leq A_{per}(y)\xi\cdot\xi\leq
\beta\left\vert \xi\right\vert ^{2}\text{ for all }\xi\in\mathbb{R}^{d}\text{
and a.e. }y\in\mathbb{R}^{d}. \label{2.2}%
\end{equation}
We denote by $L_{\infty,per}^{2}(\mathbb{R}^{d})$ the Besicovitch space
associated to the algebra with mean value $\mathcal{A}=\mathcal{B}%
_{\infty,per}(\mathbb{R}^{d}):=\mathcal{C}_{0}(\mathbb{R}^{d})\oplus
\mathcal{C}_{per}(Y)$, and by $H_{\infty,per}^{1}(\mathbb{R}^{d})$ the
corresponding Besicovitch-Sobolev type space. The next theorem provides us
with the existence of true correctors.

\begin{theorem}
\label{t1.2}Assume that $A$ and $V$ satisfy \emph{(\ref{1.2})},
\emph{(\ref{2.1})} and \emph{(\ref{2.2})}. Then the corrector problem
\begin{equation}
-\operatorname{div}(A\nabla\chi_{0}+V)=0\text{ in }\mathbb{R}^{d}\text{,
}M(\chi_{0})=0 \label{0*}%
\end{equation}
has a weak solution $\chi_{0}\in H_{\infty,per}^{1}(\mathbb{R}^{d})$. In
addition $\chi_{0}$ has the form $\chi_{0}=\chi_{0}^{0}+\chi_{per}$ where
$\chi_{per}$ is the unique solution of the periodic corrector problem
\[
-\operatorname{div}(A_{per}\nabla\chi_{per}+V_{per})=0\text{ in }%
\mathbb{R}^{d}\text{, }\int_{Y}\chi_{per}=0
\]
and $\chi_{0}^{0}\in L_{0}^{2}(\mathbb{R}^{d})$.
\end{theorem}

Theorem \ref{t1.2} has been proved in \cite{Lebris, Blanc2018, Blanc2019, BlancAA} (see
also \cite{JTW}) under the following restrictive assumption on the
coefficients $A$ and $V$:

\begin{itemize}
\item[(A1)] $A_{per}$, $V_{per}$ and $A_{0}$ have entries in $\mathcal{C}%
^{0,\nu}(\mathbb{R}^{d})$ for some $\nu>0$

\noindent and

\item[(A2)] $A_{0}$ and $V_{0}$ have entries in $L^{r}(\mathbb{R}^{d})$ for
some $1\leq r<\infty$.
\end{itemize}

In the current work, no such \ restriction is made on the coefficients $A$ and
$V$. Few comments are in order. 1) Assumption (A2) (for $r\geq2$) is a special
case of our general assumption. Indeed for any $r\geq2$, we have
$L^{r}(\mathbb{R}^{d})\subset L_{0}^{2}(\mathbb{R}^{d})$ where $L_{0}%
^{2}(\mathbb{R}^{d})$ is the closure of $\mathcal{C}_{0}(\mathbb{R}^{d})$ with
respect to the seminorm $\left\Vert u\right\Vert _{2}=(M(\left\vert
u\right\vert ^{2}))^{\frac{1}{2}}$. The latter inclusion stems from the fact
that $L_{loc}^{r}(\mathbb{R}^{d})\subset L_{loc}^{2}(\mathbb{R}^{d})$ for
$r\geq2$, so that $L^{r}(\mathbb{R}^{d})\subset L_{0}^{r}(\mathbb{R}%
^{d})\subset L_{0}^{2}(\mathbb{R}^{d})$. 2) The assumption (A1) is made
essentially in order to obtain the boundedness of the gradient of the periodic
corrector, which is a crucial step in \cite{Lebris, Blanc2018, Blanc2019, BlancAA,
JTW}. For the existence of the asymptotic periodic corrector, we need the
existence of its periodic component. This justifies the further assumption
(\ref{2.2}), which besides is also made in all these previous references.

Before stating the next result, it seems necessary to outline the proof of
Theorem \ref{t1.2}, which is short but subtle. It relies on two main features.
1) The resolution of the corrector problem (\ref{0*}) in the weak sense of
distributions in $\mathbb{R}^{d}$, which provides us with a solution $\chi
_{0}\in H_{loc}^{1}(\mathbb{R}^{d})$ such that its gradient $\nabla\chi_{0}\in
L_{\infty,per}^{2}(\mathbb{R}^{d})$ is unique. 2) The resolution of the same
problem in the sense of the duality arising from the mean value. In this way,
we first establish the existence of an isometric isomorphism between the space
of periodic correctors $H_{per}^{1}(Y)/\mathbb{R}$ and the one of asymptotic
periodic correctors $\mathcal{H}_{\#}^{1}(\mathbb{R}^{d})$, and we show that
the asymptotic periodic corrector exists if and only if the periodic one
exists. This leads to the existence of a unique class modulo $L_{0}%
^{2}(\mathbb{R}^{d})$ of the solutions of the asymptotic periodic corrector
problem. Finally we check that the solution obtained in the sense of
distributions, belongs to the unique class of solutions obtained above, so
that the existence of the solutions to (\ref{0*}) in $H_{\infty,per}%
^{1}(\mathbb{R}^{d})$ is established, and it is worth stating that this
solution is not unique. Its uniqueness is up to an additive function in
$L_{0}^{2}(\mathbb{R}^{d})$, that is, its uniqueness is closely related to the
uniqueness of the solution to its periodic component.

Relying on the existence of true correctors, we are able to prove that the
rates of convergence in $L^{2}$ are optimal provided $u_{0}\in H^{2}(\Omega)$.
The following theorem gives such a result.

\begin{theorem}
\label{t2.1}Let $\Omega$ be a $\mathcal{C}^{1,1}$ bounded domain in
$\mathbb{R}^{d}$. Suppose \emph{(\ref{1.2})}, \emph{(\ref{2.1})},
\emph{(\ref{2.2'})} and \emph{(\ref{2.2})} hold. Let $u_{\varepsilon}$ and
$u_{0}$ be the weak solutions of the Dirichlet homogeneous problems
\emph{(\ref{*10})} and \emph{(\ref{1.4'})} respectively. Suppose further that
$u_{0}\in H^{2}(\Omega)$ and $F=0$. Then, for $\varepsilon>0$ we have the
estimate
\begin{equation}
\left\Vert u_{\varepsilon}-u_{0}\right\Vert _{L^{2}(\Omega)}\leq
C\varepsilon\left\Vert f\right\Vert _{L^{2}(\Omega)} \label{5.9}%
\end{equation}
where the constant $C$ depends only on $d$, $\alpha$, $\beta$, $\alpha_{0}$
and $\Omega$.
\end{theorem}

We point out that, to our knowledge, the result in Theorem \ref{t2.1} is new
and extends the results obtained in \cite{Lebris, Blanc2018, Blanc2019, BlancAA, JTW}
to more general elliptic equations in at least two perspectives:

\begin{enumerate}
\item The operator considered in this study is more general than the one
treated in these works and contains lower order terms, and so requires a
careful treatment;

\item We do not assume neither any $L^{r}$ integrability on the asymptotic
components $A_{0}$ and $V_{0}$ of the coefficients $A$ and $V$, nor any
H\"{o}lder continuity of the periodic components $A_{per}$ and $V_{per}$ as it
is the case in all the previous references cited above. This leads to the
existence of a corrector not necessarily smooth. As a result, we use a
smoothing operator borrowed from \cite{Shen2} and that is well-suited for the
treatment of nonperiodic homogenization problems, which itself follows from an
original idea by Zhikov and Pastukhova \cite{Zhikov1} in the periodic setting
(see \cite{CPZ} for the nonlinear case).
\end{enumerate}

We also point out an important fact. With less regularity than in
\cite{Blanc2019, BlancAA}, however we obtain sharper $L^{2}$ convergence rates than in
\cite{Blanc2019, BlancAA}. Indeed in \cite{Blanc2019, BlancAA} it is shown that if the matrix
$A$ has the form $A=A_{0}+A_{per}$ with $A_{0}\in L^{r}(\mathbb{R}%
^{d})^{d\times d}$ for $r\geq2$, then it holds
\[
\left\Vert u_{\varepsilon}-u_{0}\right\Vert _{L^{2}(\Omega)}\leq
C\varepsilon^{\nu_{r}}\text{ where }\nu_{r}=\min(1,d/r).
\]
Noticing that the matrices $A$ with $A_{0}\in L^{r}(\mathbb{R}^{d})^{d\times
d}$ as given in \cite{Blanc2019, BlancAA} are special cases of our general assumption
(H3), we obtain by Theorem \ref{t2.1} and under this special hypothesis,
optimal convergence rates, in contrast with the result obtained in
\cite{Blanc2019, BlancAA}, which is not always optimal.

The homogenization of linear elliptic equations of general form has been
extensively studied in the literature. Nevertheless, to our knowledge, the
only works that are closely related to ours are \cite{Xu1, Xu2}. In these
works, the author considers a system of elliptic operators each of the same
type as $\mathcal{P}_{\varepsilon}$, but in the periodic framework. However, a
careful study of the results obtained in \cite{Xu1, Xu2} reveals that our
results set a fundamental basis for the study undertaken in these works in
nonsmooth domains, under asymptotic periodic assumption on the coefficients.

The paper is outlined as follows. In Section 2 we prove a Helmholtz-type
decomposition theorem (Theorem \ref{t1}) for the generalized Besicovitch
spaces and apply it to derive a simple proof of a compactness result related
to sigma-convergence. The well-posedness of the local corrector problems is
guaranteed by the results stated in Section 3. Finally, Sections 4, 5 and 6
are dedicated to the proof of the main theorems.

\section{On a Helmholtz-type result and application\label{sec2}}

We start this section by giving some fundamentals of algebras with mean value.

\subsection{An overview of algebras with mean value and sigma-convergence
concept}

Let $\mathcal{A}$ be an algebra with mean value on $\mathbb{R}^{d}$
\cite{Jikov}, that is, a closed subalgebra of the $\mathcal{C}^{\ast}$-algebra
of bounded uniformly continuous real-valued functions on $\mathbb{R}^{d}$,
$\mathrm{BUC}(\mathbb{R}^{d})$, which contains the constants, is translation
invariant and is such that any of its elements possesses a mean value in the
following sense: for every $u\in\mathcal{A}$, the sequence $(u^{\varepsilon
})_{\varepsilon>0}$ ($u^{\varepsilon}(x)=u(x/\varepsilon)$) weakly$\ast
$-converges in $L^{\infty}(\mathbb{R}^{d})$ to some real number $M(u)$ (called
the mean value of $u$) as $\varepsilon\rightarrow0$. The mean value expresses
as
\begin{equation}
M(u)=\lim_{R\rightarrow\infty}%
\mathchoice {{\setbox0=\hbox{$\displaystyle{\textstyle
-}{\int}$ } \vcenter{\hbox{$\textstyle -$
}}\kern-.6\wd0}}{{\setbox0=\hbox{$\textstyle{\scriptstyle -}{\int}$ } \vcenter{\hbox{$\scriptstyle -$
}}\kern-.6\wd0}}{{\setbox0=\hbox{$\scriptstyle{\scriptscriptstyle -}{\int}$
} \vcenter{\hbox{$\scriptscriptstyle -$
}}\kern-.6\wd0}}{{\setbox0=\hbox{$\scriptscriptstyle{\scriptscriptstyle
-}{\int}$ } \vcenter{\hbox{$\scriptscriptstyle -$ }}\kern-.6\wd0}}\!\int
_{B_{R}}u(y)dy\text{ for }u\in\mathcal{A}. \label{0.1}%
\end{equation}
To an algebra $\mathcal{A}$ with mean value are attached its smooth
subalgebras $\mathcal{A}^{m}=\{u\in\mathcal{A}:D^{\beta}u\in\mathcal{A}$ for
all $\left\vert \beta\right\vert \leq m\}$ and $\mathcal{A}^{\infty}=\cap
_{m}\mathcal{A}^{m}$.

For $1\leq p<\infty$, we define the Marcinkiewicz space $\mathfrak{M}%
^{p}(\mathbb{R}^{d})$ to be the set of functions $u\in L_{loc}^{p}%
(\mathbb{R}^{d})$ such that
\[
\underset{R\rightarrow\infty}{\lim\sup}%
\mathchoice {{\setbox0=\hbox{$\displaystyle{\textstyle
-}{\int}$ } \vcenter{\hbox{$\textstyle -$
}}\kern-.6\wd0}}{{\setbox0=\hbox{$\textstyle{\scriptstyle -}{\int}$ }
\vcenter{\hbox{$\scriptstyle -$
}}\kern-.6\wd0}}{{\setbox0=\hbox{$\scriptstyle{\scriptscriptstyle -}{\int}$
} \vcenter{\hbox{$\scriptscriptstyle -$
}}\kern-.6\wd0}}{{\setbox0=\hbox{$\scriptscriptstyle{\scriptscriptstyle
-}{\int}$ } \vcenter{\hbox{$\scriptscriptstyle -$ }}\kern-.6\wd0}}\!\int
_{B_{R}}\left\vert u(y)\right\vert ^{p}dy<\infty.
\]
Endowed with the seminorm $\left\Vert u\right\Vert _{p}=\left(  \underset
{R\rightarrow\infty}{\lim\sup}%
\mathchoice {{\setbox0=\hbox{$\displaystyle{\textstyle
-}{\int}$ } \vcenter{\hbox{$\textstyle -$
}}\kern-.6\wd0}}{{\setbox0=\hbox{$\textstyle{\scriptstyle -}{\int}$ }
\vcenter{\hbox{$\scriptstyle -$
}}\kern-.6\wd0}}{{\setbox0=\hbox{$\scriptstyle{\scriptscriptstyle -}{\int}$
} \vcenter{\hbox{$\scriptscriptstyle -$
}}\kern-.6\wd0}}{{\setbox0=\hbox{$\scriptscriptstyle{\scriptscriptstyle
-}{\int}$ } \vcenter{\hbox{$\scriptscriptstyle -$ }}\kern-.6\wd0}}\!\int
_{B_{R}}\left\vert u(y)\right\vert ^{p}dy\right)  ^{1/p}$, the space
$\mathfrak{M}^{p}(\mathbb{R}^{d})$ is a complete seminormed space.

\begin{definition}
\label{B-space} We define the \textit{generalized Besicovitch} \textit{space}
$B_{\mathcal{A}}^{p}(\mathbb{R}^{d})$ ($1\leq p<\infty$) as the closure of the
algebra $\mathcal{A}$ in $\mathfrak{M}^{p}(\mathbb{R}^{d})$.
\end{definition}

Since any function in $B_{\mathcal{A}}^{p}(\mathbb{R}^{d})$ is the limit of a
sequence of elements in $\mathcal{A}$, we get that for any $u\in
B_{\mathcal{A}}^{p}(\mathbb{R}^{d})$, $M(\left\vert u\right\vert ^{p})$ exists
and we have
\begin{equation}
\left\Vert u\right\Vert _{p}=\left(  \lim_{R\rightarrow\infty}%
\mathchoice {{\setbox0=\hbox{$\displaystyle{\textstyle
-}{\int}$ } \vcenter{\hbox{$\textstyle -$
}}\kern-.6\wd0}}{{\setbox0=\hbox{$\textstyle{\scriptstyle -}{\int}$ } \vcenter{\hbox{$\scriptstyle -$
}}\kern-.6\wd0}}{{\setbox0=\hbox{$\scriptstyle{\scriptscriptstyle -}{\int}$
} \vcenter{\hbox{$\scriptscriptstyle -$
}}\kern-.6\wd0}}{{\setbox0=\hbox{$\scriptscriptstyle{\scriptscriptstyle
-}{\int}$ } \vcenter{\hbox{$\scriptscriptstyle -$ }}\kern-.6\wd0}}\!\int
_{B_{R}}\left\vert u(y)\right\vert ^{p}dy\right)  ^{\frac{1}{p}}=(M(\left\vert
u\right\vert ^{p}))^{\frac{1}{p}}. \label{0.2}%
\end{equation}
In this regard, we consider the space
\[
B_{\mathcal{A}}^{1,p}(\mathbb{R}^{d})=\{u\in B_{\mathcal{A}}^{p}%
(\mathbb{R}^{d}):\nabla_{y}u\in(B_{\mathcal{A}}^{p}(\mathbb{R}^{d}))^{d}\}
\]
endowed with the seminorm
\[
\left\Vert u\right\Vert _{1,p}=\left(  \left\Vert u\right\Vert _{p}%
^{p}+\left\Vert \nabla_{y}u\right\Vert _{p}^{p}\right)  ^{\frac{1}{p}},
\]
which is a complete seminormed space. The Banach counterpart of the previous
spaces are defined as follows. We set $\mathcal{B}_{\mathcal{A}}%
^{p}(\mathbb{R}^{d})=B_{\mathcal{A}}^{p}(\mathbb{R}^{d})/\mathcal{N}$ where
$\mathcal{N}=\{u\in B_{\mathcal{A}}^{p}(\mathbb{R}^{d}):\left\Vert
u\right\Vert _{p}=0\}$. We define $\mathcal{B}_{\mathcal{A}}^{1,p}%
(\mathbb{R}^{d})$ mutatis mutandis: replace $B_{\mathcal{A}}^{p}%
(\mathbb{R}^{d})$ by $\mathcal{B}_{\mathcal{A}}^{p}(\mathbb{R}^{d})$ and
$\partial/\partial y_{i}$ by $\overline{\partial}/\partial y_{i}$, where
$\overline{\partial}/\partial y_{i}$ is defined by
\begin{equation}
\frac{\overline{\partial}}{\partial y_{i}}(u+\mathcal{N}):=\frac{\partial
u}{\partial y_{i}}+\mathcal{N}\text{ for }u\in B_{\mathcal{A}}^{1,p}%
(\mathbb{R}^{d}). \label{0.3}%
\end{equation}
It is important to note that $\overline{\partial}/\partial y_{i}$ is also
defined as the infinitesimal generator in the $i$th direction coordinate of
the strongly continuous group $\mathcal{T}(y):\mathcal{B}_{\mathcal{A}}%
^{p}(\mathbb{R}^{d})\rightarrow\mathcal{B}_{\mathcal{A}}^{p}(\mathbb{R}%
^{d});\ \mathcal{T}(y)(u+\mathcal{N})=u(\cdot+y)+\mathcal{N}$. Let us denote
by $\varrho:B_{\mathcal{A}}^{p}(\mathbb{R}^{d})\rightarrow\mathcal{B}%
_{\mathcal{A}}^{p}(\mathbb{R}^{d})=B_{\mathcal{A}}^{p}(\mathbb{R}%
^{d})/\mathcal{N}$, $\varrho(u)=u+\mathcal{N}$, the canonical surjection. We
remark that if $u\in B_{\mathcal{A}}^{1,p}(\mathbb{R}^{d})$ then
$\varrho(u)\in\mathcal{B}_{\mathcal{A}}^{1,p}(\mathbb{R}^{d})$ with further
\[
\frac{\overline{\partial}\varrho(u)}{\partial y_{i}}=\varrho\left(
\frac{\partial u}{\partial y_{i}}\right)  ,
\]
as seen above in (\ref{0.3}).

We need a further notion. A function $f\in\mathcal{B}_{\mathcal{A}}%
^{1}(\mathbb{R}^{d})$ is said to be \emph{invariant} if for any $y\in
\mathbb{R}^{d}$, $\mathcal{T}(y)f=f$. It is immediate that the above notion of
invariance is the well-known one relative to dynamical systems. An algebra
with mean value will therefore said to be \emph{ergodic} if every invariant
function $f$ is constant in $\mathcal{B}_{\mathcal{A}}^{1}(\mathbb{R}^{d})$.
As in \cite{CMP} one may show that $f\in\mathcal{B}_{\mathcal{A}}%
^{1}(\mathbb{R}^{d})$ is invariant if and only if $\frac{\overline{\partial}%
f}{\partial y_{i}}=0$ for all $1\leq i\leq d$. We denote by $I_{\mathcal{A}%
}^{p}(\mathbb{R}^{d})$ the set of $f\in\mathcal{B}_{\mathcal{A}}%
^{p}(\mathbb{R}^{d})$ that are invariant. The set $I_{\mathcal{A}}%
^{p}(\mathbb{R}^{d})$ is a closed vector subspace of $\mathcal{B}%
_{\mathcal{A}}^{p}(\mathbb{R}^{d})$ satisfying the following important
property:%
\begin{equation}
f\in I_{\mathcal{A}}^{p}(\mathbb{R}^{d})\text{ if and only if }\frac
{\overline{\partial}f}{\partial y_{i}}=0\text{ for all }1\leq i\leq d\text{.}
\label{inv}%
\end{equation}

\begin{remark}
\label{r0}\emph{It is important to note that there are algebras with mean
value that are not ergodic. A typical example is the algebra }$\mathcal{A}%
$\emph{ generated by the function }$f(y)=\cos\sqrt[3]{y}$\emph{ (}%
$y\in\mathbb{R}$\emph{) and its translates. It is shown in \cite[page
243]{Jikov} that }$\mathcal{A}$\emph{ is not ergodic. So our setting is beyond
ergodic algebras in general.}
\end{remark}

Let us also recall the following properties \cite[Proposition 2.4, Theorem 2.6
and Corollary 2.7]{CMP} (see also \cite{NA}):

\begin{itemize}
\item[(P)$_{1}$] The mean value $M$ viewed as defined on $\mathcal{A}$,
extends by continuity to a non negative continuous linear form (still denoted
by $M$) on $B_{\mathcal{A}}^{p}(\mathbb{R}^{d})$. For each $u\in
B_{\mathcal{A}}^{p}(\mathbb{R}^{d})$ and all $a\in\mathbb{R}^{d}$, we have
$M(u(\cdot+a))=M(u)$, and $\left\Vert u\right\Vert _{p}=\left[  M(\left\vert
u\right\vert ^{p})\right]  ^{1/p}$.

\item[(P)$_{2}$] Under the duality defined by $\left\langle u,v\right\rangle
=M(uv)$ for $u\in B_{\mathcal{A}}^{p}(\mathbb{R}^{d})$ and $v\in
B_{\mathcal{A}}^{p^{\prime}}(\mathbb{R}^{d})$ ($1<p<\infty$ and $\frac{1}%
{p}+\frac{1}{p^{\prime}}=1$), the topological dual of the space
$B_{\mathcal{A}}^{p}(\mathbb{R}^{d})$ is $B_{\mathcal{A}}^{p^{\prime}%
}(\mathbb{R}^{d})$. More precisely, if $L:B_{\mathcal{A}}^{p}(\mathbb{R}%
^{d})\rightarrow\mathbb{R}$ is a continuous linear functional, there exists a
function $u\in B_{\mathcal{A}}^{p^{\prime}}(\mathbb{R}^{d})$ which is unique
up to an additive function $w\in\mathcal{N}$ such that $L(v)=M(uv)$ for all
$v\in B_{\mathcal{A}}^{p}(\mathbb{R}^{d})$.
\end{itemize}

To the space $B_{\mathcal{A}}^{p}(\mathbb{R}^{d})$ we also attach the
following \textit{corrector} space
\[
B_{\#\mathcal{A}}^{1,p}(\mathbb{R}^{d})=\{u\in W_{loc}^{1,p}(\mathbb{R}%
^{d}):\nabla u\in B_{\mathcal{A}}^{p}(\mathbb{R}^{d})^{d}\text{ and }M(\nabla
u)=0\}\text{.}%
\]
In $B_{\#\mathcal{A}}^{1,p}(\mathbb{R}^{d})$ we identify two elements by their
gradients: $u=v$ in $B_{\#\mathcal{A}}^{1,p}(\mathbb{R}^{d})$ iff
$\nabla(u-v)=0$, i.e. $\left\Vert \nabla(u-v)\right\Vert _{p}=0$. We may
therefore equip $B_{\#\mathcal{A}}^{1,p}(\mathbb{R}^{d})$ with the gradient
norm $\left\Vert u\right\Vert _{\#,p}=\left\Vert \nabla u\right\Vert _{p}$.
This defines a Banach space \cite[Theorem 3.12]{Casado} containing
$B_{\mathcal{A}}^{1,p}(\mathbb{R}^{d})$ as a subspace.

We are now able to define the $\Sigma$-convergence concept. Let $1\leq
p<\infty$ be a real number, and let $\Omega$ be an open set in $\mathbb{R}%
^{d}$. A sequence $(u_{\varepsilon})_{\varepsilon>0}\subset L^{p}(\Omega)$ is
said to:

\begin{itemize}
\item[(i)] \emph{weakly }$\Sigma$\emph{-converge} in $L^{p}(\Omega)$ to
$u_{0}\in L^{p}(\Omega;\mathcal{B}_{\mathcal{A}}^{p}(\mathbb{R}^{d}))$ if, as
$\varepsilon\rightarrow0$,
\begin{equation}
\int_{\Omega}u_{\varepsilon}(x)f\left(  x,\frac{x}{\varepsilon}\right)
dx\rightarrow\int_{\Omega}M(u_{0}(x,\cdot)f(x,\cdot))dx \label{3.1}%
\end{equation}
for any $f\in L^{p^{\prime}}(\Omega;\mathcal{A})$ ($p^{\prime}=p/(p-1)$);

\item[(ii)] \emph{strongly }$\Sigma$\emph{-converge} in $L^{p}(\Omega)$ to
$u_{0}\in L^{p}(\Omega;\mathcal{B}_{\mathcal{A}}^{p}(\mathbb{R}^{d}))$ if
(\ref{3.1}) holds and further $\left\Vert u_{\varepsilon}\right\Vert
_{L^{p}(\Omega)}\rightarrow\left\Vert u_{0}\right\Vert _{L^{p}(\Omega
;\mathcal{B}_{\mathcal{A}}^{p}(\mathbb{R}^{d}))}$.
\end{itemize}

We denote (i) by \textquotedblleft$u_{\varepsilon}\rightarrow u_{0}$ in
$L^{p}(\Omega)$-weak $\Sigma$\textquotedblright, and (ii) by \textquotedblleft%
$u_{\varepsilon}\rightarrow u_{0}$ in $L^{p}(\Omega)$-strong $\Sigma
$\textquotedblright. It is to be noted that this is a generalization of the
well known concept of two-scale convergence (for which, for example,
(\ref{3.1}) holds true provided that $\mathcal{A}=\mathcal{C}_{per}(Y)$).

The following are the main properties of the above concept.

\begin{itemize}
\item[(SC)$_{1}$] Any bounded ordinary sequence in $L^{p}(\Omega)$
($1<p<\infty$) possesses a subsequence that weakly $\Sigma$-converges in
$L^{p}(\Omega)$; see e.g., \cite[Theorem 3.1]{CMP} for the justification.

\item[(SC)$_{2}$] If $(u_{\varepsilon})_{\varepsilon\in E}$ is a bounded
sequence ($E$ an ordinary sequence of positive real numbers converging to
zero) in $W^{1,p}(\Omega)$, then there exist a subsequence $E^{\prime}$ of $E$
and a couple $(u_{0},u_{1})\in W^{1,p}(\Omega;I_{\mathcal{A}}^{p})\times
L^{p}(\Omega;B_{\#\mathcal{A}}^{1,p}(\mathbb{R}^{d}))$ such that

\begin{itemize}
\item[(i)] $u_{\varepsilon}\rightarrow u_{0}$ in $L^{p}(\Omega)$-weak $\Sigma$

\item[(ii)] $\frac{\partial u_{\varepsilon}}{\partial x_{j}}\rightarrow
\frac{\partial u_{0}}{\partial x_{j}}+\frac{\partial u_{1}}{\partial y_{j}}$
in $L^{p}(\Omega)$-weak $\Sigma\ \ (1\leq j\leq d)$.
\end{itemize}

\item[(SC)$_{3}$] If $u_{\varepsilon}\rightarrow u_{0}$ in $L^{p}(\Omega
)$-weak $\Sigma$ and $v_{\varepsilon}\rightarrow v_{0}$ in $L^{q}(\Omega
)$-strong $\Sigma$, then $u_{\varepsilon}v_{\varepsilon}\rightarrow u_{0}%
v_{0}$ in $L^{r}(\Omega)$-weak $\Sigma$, where $1\leq p,q,r<\infty$ and
$\frac{1}{p}+\frac{1}{q}=\frac{1}{r}$; see \cite[Theorem 3.5]{BJMA}\ for the justification.
\end{itemize}

The proof of (SC)$_{2}$ is more involved and will be done in the next section.
It is to be noted that (SC)$_{2}$ has been proved in \cite{EJDE2014} (see
Theorem 2.11 therein) using a Helmholtz type result. For the sake of
completeness, we repeat the proof here. The reason is that in the current
work, the corrector function $u_{1}(x,\cdot)$ is found in the space
$B_{\#\mathcal{A}}^{1,p}(\mathbb{R}^{d})$ whose elements are locally
integrable functions, in contrast to the space $\mathcal{B}_{\mathcal{A}%
}^{1,p}(\mathbb{R}^{d})$ (used in \cite{EJDE2014}) whose elements are not
usual locally integrable functions.

\subsection{Helmholtz-type decomposition theorem}

Let $u\in\mathcal{A}$ (where $\mathcal{A}$ is an algebra with mean value on
$\mathbb{R}^{d}$) and let $\varphi\in\mathcal{C}_{0}^{\infty}(\mathbb{R}^{d}%
)$. Since $u$ and $\varphi$ are uniformly continuous and $\mathcal{A}$ is
translation invariant, we have $u\ast\varphi\in\mathcal{A}$ (see the proof of
Proposition 2.3 in \cite{ACAP}), where here $\ast$ stands for the usual
convolution operator. More precisely, $u\ast\varphi\in\mathcal{A}^{\infty}$
since $D_{y}^{\alpha}(u\ast\varphi)=u\ast D_{y}^{\alpha}\varphi$ for any
$\alpha\in\mathbb{N}^{d}$. For $1\leq p<\infty$, let $u\in B_{\mathcal{A}}%
^{p}(\mathbb{R}^{d})$ and $\eta>0$, and choose $v\in\mathcal{A}$ such that
$\left\Vert u-v\right\Vert _{p}<\eta/(\left\Vert \varphi\right\Vert
_{L^{1}(\mathbb{R}^{d})}+1)$. Using Young's inequality, we have
\[
\left\Vert u\ast\varphi-v\ast\varphi\right\Vert _{p}\leq\left\Vert
\varphi\right\Vert _{L^{1}(\mathbb{R}^{d})}\left\Vert u-v\right\Vert _{p}%
<\eta,
\]
hence $u\ast\varphi\in B_{\mathcal{A}}^{p}(\mathbb{R}^{d})$ since
$v\ast\varphi\in\mathcal{A}$. Moreover, we have
\begin{equation}
\left\Vert u\ast\varphi\right\Vert _{p}\leq\left\vert \mathrm{supp}%
\varphi\right\vert ^{\frac{1}{p}}\left\Vert \varphi\right\Vert _{L^{p^{\prime
}}(\mathbb{R}^{d})}\left\Vert u\right\Vert _{p}, \label{2}%
\end{equation}
where \textrm{supp}$\varphi$ stands for the support of $\varphi$ and
$\left\vert \mathrm{supp}\varphi\right\vert $ its Lebesgue measure. Indeed, we
have
\[
\left\Vert u\ast\varphi\right\Vert _{p}=\left(  \underset{r\rightarrow+\infty
}{\lim\sup}\mathchoice {{\setbox0=\hbox{$\displaystyle{\textstyle
-}{\int}$ } \vcenter{\hbox{$\textstyle -$
}}\kern-.6\wd0}}{{\setbox0=\hbox{$\textstyle{\scriptstyle -}{\int}$ }
\vcenter{\hbox{$\scriptstyle -$
}}\kern-.6\wd0}}{{\setbox0=\hbox{$\scriptstyle{\scriptscriptstyle -}{\int}$
} \vcenter{\hbox{$\scriptscriptstyle -$
}}\kern-.6\wd0}}{{\setbox0=\hbox{$\scriptscriptstyle{\scriptscriptstyle
-}{\int}$ } \vcenter{\hbox{$\scriptscriptstyle -$ }}\kern-.6\wd0}}\!\int
_{B_{r}}\left\vert (u\ast\varphi)(y)\right\vert ^{p}dy\right)  ^{\frac{1}{p}%
},
\]
and
\begin{align*}
\int_{B_{r}}\left\vert (u\ast\varphi)(y)\right\vert ^{p}dy  &  =\int_{B_{r}%
}\left\vert \int_{\mathbb{R}^{d}}u(y-t)\varphi(t)dt\right\vert ^{p}dy\\
&  \leq\int_{B_{r}}\left(  \int_{\mathbb{R}^{d}}\left\vert \varphi
(t)\right\vert ^{1/p^{\prime}}\left\vert \varphi(t)\right\vert ^{1/p}%
\left\vert u(y-t)\right\vert dt\right)  ^{p}dy\\
&  \leq\left(  \int_{\mathbb{R}^{d}}\left\vert \varphi(t)\right\vert
dt\right)  ^{p/p^{\prime}}\int_{B_{r}}\int_{\mathbb{R}^{d}}\left\vert
\varphi(t)\right\vert \left\vert u(y-t)\right\vert ^{p}dtdy\\
&  =\left(  \int_{\mathbb{R}^{d}}\left\vert \varphi(t)\right\vert dt\right)
^{p/p^{\prime}}\int_{\mathbb{R}^{d}}\left\vert \varphi(t)\right\vert \left(
\int_{B_{r}}\left\vert u(y-t)\right\vert ^{p}dy\right)  dt.
\end{align*}
Dividing by $\left\vert B_{r}\right\vert $ and computing the $\lim
\sup_{r\rightarrow\infty}$, we obtain
\[
\left\Vert u\ast\varphi\right\Vert _{p}^{p}\leq\left(  \int_{\mathbb{R}^{d}%
}\left\vert \varphi(t)\right\vert dt\right)  ^{p/p^{\prime}}\int
_{\mathbb{R}^{d}}\left\vert \varphi(t)\right\vert \left\Vert u(\cdot
-t)\right\Vert _{p}^{p}dt.
\]
The inequality (\ref{2}) is therefore a consequence of the invariance under
translations, of the seminorm $\left\Vert \cdot\right\Vert _{p}$: $\left\Vert
u(\cdot-t)\right\Vert _{p}=\left\Vert u\right\Vert _{p}$, for all
$t\in\mathbb{R}^{d}$. The following Helmholtz-type result holds.

\begin{theorem}
\label{t1}Let $1<p<\infty$. Let $L$ be a bounded linear functional on
$(B_{\mathcal{A}}^{1,p^{\prime}}(\mathbb{R}^{d}))^{d}$ which vanishes on the
kernel of the divergence. Then there exists a function $f\in B_{\mathcal{A}%
}^{p}(\mathbb{R}^{d})$ such that $L=\nabla_{y}f$, in the sense that
\[
L(v)=-M(f\Div v)\text{ for all }v\in(B_{\mathcal{A}}^{1,p^{\prime}}%
(\mathbb{R}^{d}))^{d}\text{.}%
\]
Moreover $f+\mathcal{N}$ is unique modulo $I_{\mathcal{A}}^{p}$, that is, up
to an additive function $g\in B_{\mathcal{A}}^{p}(\mathbb{R}^{d})$ verifying
$\overline{\nabla}_{y}\varrho(g)=0$.
\end{theorem}

\begin{proof}
Fix $u$ in $\mathcal{A}^{\infty}$ and define $L_{u}:\mathcal{D}(\mathbb{R}%
^{d})^{d}\rightarrow\mathbb{R}$ by
\[
L_{u}(\varphi)=L(u\ast\varphi)\text{ for }\varphi=(\varphi_{i})\in
\mathcal{D}(\mathbb{R}^{d})^{d}%
\]
where $u\ast\varphi=(u\ast\varphi_{i})_{i}\in(\mathcal{A}^{\infty})^{d}$. Then
it is easy to see from (\ref{2}) that $L_{u}\in\mathcal{D}^{\prime}%
(\mathbb{R}^{d})^{d}$. Moreover if $\Div\varphi=0$ then $\Div(u\ast
\varphi)=u\ast\Div\varphi=0$, so that $L_{u}(\varphi)=0$, that is, $L_{u}$
vanishes on the kernel of the divergence in $\mathcal{D}(\mathbb{R}^{d})^{d}$.
By the classical De Rham theorem (see \cite[Theorem 17', page 114]{De Rham} or
\cite{Simon}), there exists a distribution $S(u)\in\mathcal{D}^{\prime
}(\mathbb{R}^{d})$ such that
\begin{equation}
L_{u}=\nabla S(u). \label{6.1}%
\end{equation}
We derive an operator
\[
S:\mathcal{A}^{\infty}\rightarrow\mathcal{D}^{\prime}(\mathbb{R}%
^{d});\ u\mapsto S(u)
\]
satisfying:

\begin{itemize}
\item[(i)] $S(\tau_{y}u)=\tau_{-y}S(u)$ for all $y\in\mathbb{R}^{d}$ and all
$u\in\mathcal{A}^{\infty}$, where $\tau_{y}v=v(\cdot+y)$;

\item[(ii)] $S$ maps linearly and continuously $\mathcal{A}^{\infty}$ into
$L_{loc}^{p^{\prime}}(\mathbb{R}^{d})$;

\item[(iii)] There is a positive constant $C_{r}$ (that is locally bounded as
a function of $r$) such that
\[
\left\Vert S(u)\right\Vert _{L^{p^{\prime}}(B_{r})}\leq C_{r}\left\Vert
L\right\Vert \left\vert B_{r}\right\vert ^{\frac{1}{p^{\prime}}}\left\Vert
u\right\Vert _{p^{\prime}}\ \ \forall r>0.
\]

\end{itemize}

The property (i) stems from the obvious equality
\begin{equation}
L_{\tau_{y}u}(\varphi)=L_{u}(\tau_{y}\varphi)\ \ \forall y\in\mathbb{R}^{d}.
\label{6.2}%
\end{equation}
Indeed we have
\begin{align*}
L_{\tau_{y}u}(\varphi)  &  =\left\langle \nabla S(\tau_{y}u),\varphi
\right\rangle \text{ \ (by (\ref{6.1}))}\\
&  =-\left\langle S(\tau_{y}u),\Div\varphi\right\rangle \\
&  =L_{u}(\tau_{y}\varphi)\text{ \ (by (\ref{6.2}))}\\
&  =\left\langle \nabla S(u),\tau_{y}\varphi\right\rangle =-\left\langle
S(u),\Div(\tau_{-y}\varphi)\right\rangle \\
&  =-\left\langle S(u),\tau_{y}\Div\varphi\right\rangle =-\left\langle
\tau_{-y}S(u),\Div\varphi\right\rangle \\
&  =\left\langle \nabla(\tau_{-y}S(u)),\varphi\right\rangle .
\end{align*}
It follows that $\nabla(\tau_{-y}S(u))=\nabla S(\tau_{y}u)$, that is,
$S(\tau_{y}u)-\tau_{-y}S(u)$ is constant. Actually $S(\tau_{y}u)-\tau
_{-y}S(u)=0$. To see this, let $g\in\mathcal{C}_{0}^{\infty}(B_{r})$ with
$\int_{B_{r}}gdy=0$; then by \cite[Lemma 3.15]{Novotny} there exists
$\varphi\in\mathcal{C}_{0}^{\infty}(B_{r})^{d}$ such that $\Div\varphi=g$,
thus, from the above series of equalities,
\[
L_{\tau_{y}u}(\varphi)=-\left\langle S(\tau_{y}u),g\right\rangle
=-\left\langle \tau_{-y}S(u),g\right\rangle ,
\]
that is, $\left\langle S(\tau_{y}u)-\tau_{-y}S(u),g\right\rangle =0$, so that
the above constant is zero. Hence
\[
S(\tau_{y}u)=\tau_{-y}S(u)\ \ \forall y\in\mathbb{R}^{d}.
\]
Hence (i) follows thereby.

Let us now check (ii) and (iii). We begin by noticing that $S$ is trivially
linear. With this in mind, let $\varphi\in\mathcal{D}(\mathbb{R}^{d})^{d}$
with \textrm{supp}$\varphi_{i}\subset B_{r}$ for all $1\leq i\leq d$. Then by
(\ref{2})
\begin{align*}
\left\vert L_{u}(\varphi)\right\vert  &  =\left\vert L(u\ast\varphi
)\right\vert \\
&  \leq\left\Vert L\right\Vert \left\Vert u\ast\varphi\right\Vert
_{(\mathcal{B}_{\mathcal{A}}^{1,p^{\prime}})^{d}}\\
&  \leq\max_{1\leq i\leq d}\left\vert \mathrm{supp}\varphi_{i}\right\vert
^{\frac{1}{p^{\prime}}}\left\Vert L\right\Vert \left\Vert u\right\Vert
_{p^{\prime}}\left\Vert \varphi\right\Vert _{W^{1,p}(B_{r})^{d}}.
\end{align*}
Hence, as Supp$\varphi_{i}\subset B_{r}$ ($1\leq i\leq d$),
\begin{equation}
\left\Vert L_{u}\right\Vert _{W^{-1,p^{\prime}}(B_{r})^{d}}\leq\left\Vert
L\right\Vert \left\vert B_{r}\right\vert ^{\frac{1}{p^{\prime}}}\left\Vert
u\right\Vert _{p^{\prime}}. \label{3}%
\end{equation}
As above, let $g\in\mathcal{C}_{0}^{\infty}(B_{r})$ with $\int_{B_{r}}gdy=0$;
then by \cite[Lemma 3.15]{Novotny} there exists $\varphi\in\mathcal{C}%
_{0}^{\infty}(B_{r})^{d}$ such that $\Div\varphi=g$ and $\left\Vert
\varphi\right\Vert _{W^{1,p}(B_{r})^{d}}\leq C(p,B_{r})\left\Vert g\right\Vert
_{L^{p}(B_{r})}$. We have
\begin{align*}
\left\vert \left\langle S(u),g\right\rangle \right\vert  &  =\left\vert
-\left\langle \nabla S(u),\varphi\right\rangle \right\vert =\left\vert
\left\langle L_{u},\varphi\right\rangle \right\vert \\
&  \leq\left\Vert L_{u}\right\Vert _{W^{-1,p^{\prime}}(B_{r})^{d}}\left\Vert
\varphi\right\Vert _{W^{1,p}(B_{r})^{d}}\\
&  \leq C(p,B_{r})\left\Vert L\right\Vert \left\vert B_{r}\right\vert
^{\frac{1}{p^{\prime}}}\left\Vert u\right\Vert _{p^{\prime}}\left\Vert
g\right\Vert _{L^{p}(B_{r})},
\end{align*}
and by a density argument, we get that $S(u)\in(L^{p}(B_{r})/\mathbb{R}%
)^{\prime}=L^{p^{\prime}}(B_{r})/\mathbb{R}$ for any $r>0$, where
$L^{p^{\prime}}(B_{r})/\mathbb{R}=\{\psi\in L^{p^{\prime}}(B_{r}):\int_{B_{r}%
}\psi dy=0\}$. The properties (ii) and (iii) therefore follow from the above
series of inequalities.

In view of (ii) one has
\begin{equation}
L_{u}(\varphi)=-\int_{\mathbb{R}^{d}}S(u)\Div\varphi dy\text{ for all }%
\varphi\in\mathcal{D}(\mathbb{R}^{d})^{d}. \label{4}%
\end{equation}

\begin{claim}
We claim that $S(u)\in\mathcal{C}^{\infty}(\mathbb{R}^{d})$ for all
$u\in\mathcal{A}^{\infty}$.
\end{claim}

Indeed let $e_{i}=(\delta_{ij})_{1\leq j\leq d}$ ($\delta_{ij}$ the Kronecker
delta). Then owing to (i), (ii) and (iii) above, we have
\begin{align*}
\left\Vert \frac{S(u)(\cdot+te_{i})-S(u)}{t}-S\left(  \frac{\partial
u}{\partial y_{i}}\right)  \right\Vert _{L^{p^{\prime}}(B_{r})}  &
=\left\Vert S\left(  \frac{u(\cdot+te_{i})-u}{t}-\frac{\partial u}{\partial
y_{i}}\right)  \right\Vert _{L^{p^{\prime}}(B_{r})}\\
&  \leq c\left\Vert \frac{u(\cdot+te_{i})-u}{t}-\frac{\partial u}{\partial
y_{i}}\right\Vert _{p^{\prime}}.
\end{align*}
Hence, taking the limit as $t\rightarrow0$ in the above inequality yields
\[
\frac{\partial}{\partial y_{i}}S(u)=S\left(  \frac{\partial u}{\partial y_{i}%
}\right)  \text{ for all }1\leq i\leq d.
\]
Repeating the process above, it emerges
\[
D_{y}^{\alpha}S(u)=S(D_{y}^{\alpha}u)\text{ for all }\alpha\in\mathbb{N}%
^{d}\text{.}%
\]
So all the weak derivatives of $S(u)$ of any order belong to $L_{loc}%
^{p^{\prime}}(\mathbb{R}^{d})$. Our claim is therefore a consequence of
\cite[Theorem XIX, p. 191]{LS}.

This being so, we derive from the mean value theorem the existence of $\xi\in
B_{r}$ such that
\[
S(u)(\xi)=\mathchoice {{\setbox0=\hbox{$\displaystyle{\textstyle
-}{\int}$ } \vcenter{\hbox{$\textstyle -$
}}\kern-.6\wd0}}{{\setbox0=\hbox{$\textstyle{\scriptstyle -}{\int}$ }
\vcenter{\hbox{$\scriptstyle -$
}}\kern-.6\wd0}}{{\setbox0=\hbox{$\scriptstyle{\scriptscriptstyle -}{\int}$
} \vcenter{\hbox{$\scriptscriptstyle -$
}}\kern-.6\wd0}}{{\setbox0=\hbox{$\scriptscriptstyle{\scriptscriptstyle
-}{\int}$ } \vcenter{\hbox{$\scriptscriptstyle -$ }}\kern-.6\wd0}}\!\int
_{B_{r}}S(u)dy.
\]
On the other hand, the map $u\mapsto S(u)(0)$ is a linear functional on
$\mathcal{A}^{\infty}$, and by the last equality above and by (iii), we get
\begin{align*}
\left\vert S(u)(0)\right\vert  &  \leq\underset{r\rightarrow0}{\lim\sup
}\mathchoice {{\setbox0=\hbox{$\displaystyle{\textstyle
-}{\int}$ } \vcenter{\hbox{$\textstyle -$
}}\kern-.6\wd0}}{{\setbox0=\hbox{$\textstyle{\scriptstyle -}{\int}$ } \vcenter{\hbox{$\scriptstyle -$
}}\kern-.6\wd0}}{{\setbox0=\hbox{$\scriptstyle{\scriptscriptstyle -}{\int}$
} \vcenter{\hbox{$\scriptscriptstyle -$
}}\kern-.6\wd0}}{{\setbox0=\hbox{$\scriptscriptstyle{\scriptscriptstyle
-}{\int}$ } \vcenter{\hbox{$\scriptscriptstyle -$ }}\kern-.6\wd0}}\!\int
_{B_{r}}\left\vert S(u)\right\vert dy\\
&  \leq\underset{r\rightarrow0}{\lim\sup}\left\vert B_{r}\right\vert
^{-\frac{1}{p^{\prime}}}\left(  \int_{B_{r}}\left\vert S(u)\right\vert
^{p^{\prime}}dy\right)  ^{\frac{1}{p^{\prime}}}\\
&  \leq c\left\Vert L\right\Vert \left\Vert u\right\Vert _{p^{\prime}}.
\end{align*}
Hence, defining $\widetilde{S}:\mathcal{A}^{\infty}\rightarrow\mathbb{R}$ by
$\widetilde{S}(v)=S(v)(0)$ for $v\in\mathcal{A}^{\infty}$, we get that
$\widetilde{S}$ is a linear functional on $\mathcal{A}^{\infty}$ satisfying
\begin{equation}
\left\vert \widetilde{S}(v)\right\vert \leq c\left\Vert L\right\Vert
\left\Vert v\right\Vert _{p^{\prime}}\ \ \forall v\in\mathcal{A}^{\infty
}\text{.} \label{5}%
\end{equation}
We infer from both the density of $\mathcal{A}^{\infty}$ in $B_{\mathcal{A}%
}^{p^{\prime}}(\mathbb{R}^{d})$ and (\ref{5}) the \ existence of a function
$f\in B_{\mathcal{A}}^{p}(\mathbb{R}^{d})$ (recall that the topological dual
of $B_{\mathcal{A}}^{p^{\prime}}(\mathbb{R}^{d})$ is $B_{\mathcal{A}}%
^{p}(\mathbb{R}^{d})$ for $1<p<\infty$, the duality here being defined by
$\left\langle u,v\right\rangle =M(uv)$) with $\left\Vert f\right\Vert _{p}\leq
c\left\Vert L\right\Vert $ such that
\[
\widetilde{S}(v)=M(fv)\text{ for all }v\in B_{\mathcal{A}}^{p^{\prime}%
}(\mathbb{R}^{d})\text{.}%
\]
In particular
\[
S(u)(0)=M(fu)\ \ \forall u\in\mathcal{A}^{\infty}.
\]
Now, let $u\in\mathcal{A}^{\infty}$ and let $y\in\mathbb{R}^{d}$. By (i) we
have
\[
S(u)(y)=S(\tau_{-y}u)(0)=M(u(\cdot-y)f)\text{.}%
\]
Thus,
\begin{align*}
L_{u}(\varphi)  &  =L(u\ast\varphi)=-\int_{\mathbb{R}^{d}}S(u)(y)\Div\varphi
dy\text{ \ (by (\ref{4}))}\\
&  =-\int_{\mathbb{R}^{d}}\left(  M(f\tau_{-y}u)\right)  \Div\varphi dy\\
&  =-M\left(  \left[  \int_{\mathbb{R}^{d}}\tau_{-y}u(\cdot)\Div\varphi
dy\right]  f(\cdot)\right) \\
&  =-M\left(  f\left[  u\ast\Div\varphi\right]  \right) \\
&  =-M\left(  f\left[  \Div(u\ast\varphi)\right]  \right) \\
&  =\left\langle \nabla f,u\ast\varphi\right\rangle .
\end{align*}
Finally let $v\in(B_{\mathcal{A}}^{1,p^{\prime}}(\mathbb{R}^{d}))^{d}$ and let
$(\varphi_{n})_{n}\subset\mathcal{D}(\mathbb{R}^{d})$ be a mollifier. Then
$v\ast\varphi_{n}\rightarrow v$ in $(B_{\mathcal{A}}^{1,p^{\prime}}%
(\mathbb{R}^{d}))^{d}$ as $n\rightarrow\infty$, where $v\ast\varphi_{n}%
=(v_{i}\ast\varphi_{n})_{1\leq i\leq d}$. We have $v\ast\varphi_{n}%
\in(\mathcal{A}^{\infty})^{d}$ and $L(v\ast\varphi_{n})\rightarrow L(v)$ by
the continuity of $L$. On the other hand
\[
M\left(  f\left[  \Div(v\ast\varphi_{n})\right]  \right)  \rightarrow M\left(
f\Div v\right)  .
\]
We deduce that $L$ and $\nabla_{y}f$ agree on $(B_{\mathcal{A}}^{1,p^{\prime}%
}(\mathbb{R}^{d}))^{d}$, i.e., $L=\nabla_{y}f$.

For the uniqueness, let $f_{1}$ and $f_{2}$ in $B_{\mathcal{A}}^{p}%
(\mathbb{R}^{d})$ be such that $L=\nabla_{y}f_{1}=\nabla_{y}f_{2}$, then
$\overline{\nabla}_{y}(\varrho(f_{1}-f_{2}))=0$, which means that $f_{1}%
-f_{2}+\mathcal{N}\in I_{\mathcal{A}}^{p}$.
\end{proof}

As a result of Theorem \ref{t1}, we have the

\begin{corollary}
\label{c1}Let $\mathbf{f}+\mathcal{N}\in(\mathcal{B}_{\mathcal{A}}%
^{p}(\mathbb{R}^{d}))^{d}$ be such that
\[
M(\mathbf{f}\cdot\mathbf{g})=0\ \text{ }\forall\mathbf{g}\in(\mathcal{A}%
^{\infty})^{d}\text{ with }\Div\mathbf{g}=0.
\]
Then there exists a function $u\in B_{\#\mathcal{A}}^{1,p}(\mathbb{R}^{d})$
whose class $u+\mathcal{N}$ modulo $I_{\mathcal{A}}^{p}$ is unique, such that
$\mathbf{f}=\nabla u$.
\end{corollary}

\begin{proof}
Define $L:(B_{\mathcal{A}}^{1,p^{\prime}}(\mathbb{R}^{d}))^{d}\rightarrow
\mathbb{R}$ by $L(\mathbf{v})=M(\mathbf{f}\cdot\mathbf{v})$. Then $L$ lies in
$\left[  (B_{\mathcal{A}}^{1,p^{\prime}}(\mathbb{R}^{d}))^{d}\right]
^{\prime}$, and it follows from Theorem \ref{t1} the existence of $u\in
B_{\mathcal{A}}^{p}(\mathbb{R}^{d})$ such that $\mathbf{f}=\nabla_{y}u$. This
shows that $u\in B_{\#\mathcal{A}}^{1,p}(\mathbb{R}^{d})$. Indeed
$B_{\mathcal{A}}^{p}(\mathbb{R}^{d})\subset L_{loc}^{p}(\mathbb{R}^{d})$, so
that $u\in W_{loc}^{1,p}(\mathbb{R}^{d})$ and $\nabla_{y}u=\mathbf{f}%
\in(B_{\mathcal{A}}^{p}(\mathbb{R}^{d}))^{d}$. Moreover let $\varphi
\in\mathcal{D}(\mathbb{R}^{d})$ be such that $\int_{\mathbb{R}^{d}}\varphi
dy=1$; then $M(\nabla_{y}u\ast\varphi)=M(\nabla_{y}u)\int_{\mathbb{R}^{d}%
}\varphi dy=M(\nabla_{y}u)$. However $M(\nabla_{y}u\ast\varphi)=M(u\ast
\nabla_{y}\varphi)=M(u)\int_{\mathbb{R}^{d}}\nabla_{y}\varphi dy=0$. The
uniqueness is shown as in Theorem \ref{t1}.
\end{proof}

\subsection{One application to sigma-convergence}

Our aim in this subsection is to prove the compactness result (SC)$_{2}$
stated in Section \ref{sec2}. Throughout this section, $\Omega$ is an open
subset of $\mathbb{R}^{d}$, and unless otherwise specified, $\mathcal{A}$ is
an algebra with mean value on $\mathbb{R}^{d}$. The following result holds true.

\begin{theorem}
\label{t2.3}Let $1<p<\infty$. Let $(u_{\varepsilon})_{\varepsilon\in E}$ be a
bounded sequence in $W^{1,p}(\Omega)$. Then there exist a subsequence
$E^{\prime}$ of $E$, and a couple $(u_{0},u_{1})\in W^{1,p}(\Omega
;I_{\mathcal{A}}^{p})\times L^{p}(\Omega;B_{\#\mathcal{A}}^{1,p}%
(\mathbb{R}^{d}))$ such that, as $E^{\prime}\ni\varepsilon\rightarrow0$,
\[
u_{\varepsilon}\rightarrow u_{0}\ \text{in }L^{p}(\Omega)\text{-weak }%
\Sigma\text{;}%
\]%
\[
\frac{\partial u_{\varepsilon}}{\partial x_{i}}\rightarrow\frac{\partial
u_{0}}{\partial x_{i}}+\frac{\partial u_{1}}{\partial y_{i}}\text{\ in }%
L^{p}(\Omega)\text{-weak }\Sigma\text{, }1\leq i\leq d.
\]

\end{theorem}

\begin{proof}
Since the sequences $(u_{\varepsilon})_{\varepsilon\in E}$ and $(\nabla
u_{\varepsilon})_{\varepsilon\in E}$ are bounded respectively in $L^{p}%
(\Omega)$ and in $L^{p}(\Omega)^{d}$, there exist (by (SC)$_{1}$) a
subsequence $E^{\prime}$ of $E$ and $u_{0}\in L^{p}(\Omega;\mathcal{B}%
_{\mathcal{A}}^{p}(\mathbb{R}^{d}))$, $v=(v_{j})_{j}\in L^{p}(\Omega
;\mathcal{B}_{\mathcal{A}}^{p}(\mathbb{R}^{d}))^{d}$ such that $u_{\varepsilon
}\rightarrow u_{0}\ $in $L^{p}(\Omega)$-weak $\Sigma$ and $\frac{\partial
u_{\varepsilon}}{\partial x_{j}}\rightarrow v_{j}$ in $L^{p}(\Omega)$-weak
$\Sigma$. For $\Phi\in(\mathcal{C}_{0}^{\infty}(\Omega)\otimes\mathcal{A}%
^{\infty})^{d}$ and recalling that $\Phi^{\varepsilon}(x)=\Phi(x,x/\varepsilon
)$, we have
\[
\int_{\Omega}\varepsilon\nabla u_{\varepsilon}\cdot\Phi^{\varepsilon}%
dx=-\int_{\Omega}\left(  u_{\varepsilon}(\Div_{y}\Phi)^{\varepsilon
}+\varepsilon u_{\varepsilon}(\Div\Phi)^{\varepsilon}\right)  dx.
\]
Letting $E^{\prime}\ni\varepsilon\rightarrow0$ we get%
\[
-\int_{\Omega}M\left(  u_{0}(x,\cdot)\Div_{y}\Phi(x,\cdot)\right)  dx=0.
\]
This shows that $\overline{\nabla}_{y}u_{0}(x,\cdot)=0$, which means that
$u_{0}(x,\cdot)\in I_{\mathcal{A}}^{p}$, so that $u_{0}\in L^{p}%
(\Omega;I_{\mathcal{A}}^{p})$. Next let $\Phi_{\varepsilon}(x)=\varphi
(x)\Psi(x/\varepsilon)$ ($x\in\Omega$) with $\varphi\in\mathcal{C}_{0}%
^{\infty}(\Omega)$ and $\Psi=(\psi_{j})_{1\leq j\leq d}\in(\mathcal{A}%
^{\infty})^{d}$ with ${\Div}_{y}\Psi=0$. Clearly%
\[
\sum_{j=1}^{d}\int_{\Omega}\frac{\partial u_{\varepsilon}}{\partial x_{j}%
}\varphi\psi_{j}^{\varepsilon}dx=-\sum_{j=1}^{d}\int_{\Omega}u_{\varepsilon
}\psi_{j}^{\varepsilon}\frac{\partial\varphi}{\partial x_{j}}dx
\]
where $\psi_{j}^{\varepsilon}(x)=\psi_{j}(x/\varepsilon)$. Passing to the
limit in the above equation when $E^{\prime}\ni\varepsilon\rightarrow0$ we
get
\begin{equation}
\sum_{j=1}^{d}\int_{\Omega}M\left(  v_{j}\psi_{j}\right)  \varphi
dx=-\sum_{j=1}^{d}\int_{\Omega}M\left(  u_{0}\psi_{j}\right)  \frac
{\partial\varphi}{\partial x_{j}}dx. \label{2.3'}%
\end{equation}
Choosing in (\ref{2.3'}) $\psi_{j}=\delta_{ij}$ (the Kronecker delta), we
obtain, for any $1\leq i\leq d$%
\begin{equation}
\int_{\Omega}M(v_{i}\psi)\varphi dx=-\int_{\Omega}M(u_{0}\psi)\frac
{\partial\varphi}{\partial x_{i}}dx\text{ \ for all }\varphi\in\mathcal{C}%
_{0}^{\infty}(\Omega), \label{2.4'}%
\end{equation}
and reminding that $M(v_{i})\in L^{p}(\Omega)$ we have by (\ref{2.4'}) that
$\frac{\partial u_{0}}{\partial x_{i}}\in L^{p}(\Omega;I_{\mathcal{A}}^{p})$,
where $\frac{\partial u_{0}}{\partial x_{i}}$ is the distributional derivative
of $u_{0}$ with respect to $x_{i}$. We deduce that $u_{0}\in W^{1,p}%
(\Omega;I_{\mathcal{A}}^{p})$. Coming back to (\ref{2.3'}) we get
\[
\int_{\Omega}M((\mathbf{v}-\nabla u_{0})\cdot\Psi)\varphi dx=0\text{ \ for all
}\varphi\in\mathcal{C}_{0}^{\infty}(\Omega)\text{.}%
\]
It follows that
\[
M\left(  \left(  \mathbf{v}(x,\cdot)-\nabla u_{0}(x,\cdot)\right)  \cdot
\Psi\right)  =0
\]
for all $\Psi$ as above and for a.e. $x$. We infer from Corollary \ref{c1} the
existence of a function $u_{1}(x,\cdot)\in B_{\#\mathcal{A}}^{1,p}%
(\mathbb{R}^{d})$ such that
\[
\mathbf{v}(x,\cdot)-\nabla u_{0}(x,\cdot)=\nabla_{y}u_{1}(x,\cdot)
\]
for a.e. $x$. Whence the existence of a function $u_{1}:x\mapsto u_{1}%
(x,\cdot)$ with values in $B_{\#\mathcal{A}}^{1,p}(\mathbb{R}^{d})$ such that
$\mathbf{v}=\nabla u_{0}+\nabla_{y}u_{1}$.
\end{proof}

\begin{remark}
\label{r2.1'}\emph{It is very important to notice that the function }%
$u_{1}(x,\cdot)$\emph{ in the above theorem is such that its gradient }%
$\nabla_{y}u_{1}(x,\cdot)$\emph{ stands for the representative of its class
}$\nabla_{y}u_{1}(x,\cdot)+\mathcal{N}$\emph{ in }$\mathcal{B}_{\mathcal{A}%
}^{p}(\mathbb{R}^{d})^{d}$\emph{. So its uniqueness is in terms of its
equivalence class. However we will see in practice that we only need a
representative of the class, not all the class. This will come to light in the
resolution of the corrector problem (see Section \ref{sec3'} below) where we
will see that the uniqueness of }$u_{1}(x,\cdot)$\emph{ in }$B_{\#\mathcal{A}%
}^{1,p}(\mathbb{R}^{d})$\emph{ will be sufficient.}
\end{remark}

\begin{remark}
\label{r2.2'}\emph{If we assume the algebra }$\mathcal{A}$\emph{\ to be
ergodic, then }$I_{\mathcal{A}}^{p}$\emph{\ consists of constant functions, so
that the function }$u_{0}$\emph{\ in Theorem \ref{t2.3} does not depend on
}$y$\emph{, that is, }$u_{0}\in W^{1,p}(\Omega)$\emph{. We thus recover the
already known result proved in \cite{NA} in the case of ergodic algebras. In
this case, part (i) in (SC)}$_{2}$\emph{ is replaced by }

\begin{itemize}
\item[(i)$_{1}$] $u_{\varepsilon}\rightarrow u_{0}$\emph{ in }$W^{1,p}%
(\Omega)$\emph{-weak.}
\end{itemize}
\end{remark}

\section{Existence result for the corrector problem\label{sec3'}}

Let the matrix $A$ satisfy assumptions (H1) and (H2) of Section \ref{sec1}.
Our aim in this section is to prove the following corrector result.

\begin{theorem}
\label{t4.1}Let $H\in B_{\mathcal{A}}^{2}(\mathbb{R}^{d})^{d}$. Then there
exists a unique function $u\in B_{\#\mathcal{A}}^{1,2}(\mathbb{R}^{d})$ such
that
\begin{equation}
-\operatorname{div}A\nabla u=\operatorname{div}H\text{ in }\mathbb{R}^{d}.
\label{4.1}%
\end{equation}

\end{theorem}

The proof of Theorem \ref{t4.1} is a consequence of the following auxiliary results.

\begin{proposition}
\label{p4.0}Let $h\in L_{loc}^{2}(\mathbb{R}^{d})$ and $H=(h_{1},...,h_{d})\in
L_{loc}^{2}(\mathbb{R}^{d})^{d}$. Assume that
\begin{equation}
\sup_{x\in\mathbb{R}^{d}}\int_{B(x,1)}\left(  \left\vert h\right\vert
^{2}+\left\vert H\right\vert ^{2}\right)  dy<\infty. \label{6.5}%
\end{equation}
Then for any $T>0$, there exists a unique $u\in H_{loc}^{1}(\mathbb{R}^{d})$
such that
\begin{equation}
-\operatorname{div}A\nabla u+T^{-2}u=h+\operatorname{div}H\text{ in
}\mathbb{R}^{d} \label{6.6}%
\end{equation}
and
\begin{equation}
\sup_{x\in\mathbb{R}^{d}}\int_{B(x,1)}\left(  \left\vert u\right\vert
^{2}+\left\vert \nabla u\right\vert ^{2}\right)  dy<\infty. \label{6.7}%
\end{equation}
Furthermore the solution $u$ satisfies the estimate
\begin{equation}
\sup_{x\in\mathbb{R}^{d},R\geq T}%
\mathchoice {{\setbox0=\hbox{$\displaystyle{\textstyle
-}{\int}$ } \vcenter{\hbox{$\textstyle -$
}}\kern-.6\wd0}}{{\setbox0=\hbox{$\textstyle{\scriptstyle -}{\int}$ } \vcenter{\hbox{$\scriptstyle -$
}}\kern-.6\wd0}}{{\setbox0=\hbox{$\scriptstyle{\scriptscriptstyle -}{\int}$
} \vcenter{\hbox{$\scriptscriptstyle -$
}}\kern-.6\wd0}}{{\setbox0=\hbox{$\scriptscriptstyle{\scriptscriptstyle
-}{\int}$ } \vcenter{\hbox{$\scriptscriptstyle -$ }}\kern-.6\wd0}}\!\int
_{B(x,R)}\left(  T^{-2}\left\vert u\right\vert ^{2}+\left\vert \nabla
u\right\vert ^{2}\right)  dy\leq C\sup_{x\in\mathbb{R}^{d},R\geq
T}\mathchoice {{\setbox0=\hbox{$\displaystyle{\textstyle
-}{\int}$ } \vcenter{\hbox{$\textstyle -$
}}\kern-.6\wd0}}{{\setbox0=\hbox{$\textstyle{\scriptstyle -}{\int}$ } \vcenter{\hbox{$\scriptstyle -$
}}\kern-.6\wd0}}{{\setbox0=\hbox{$\scriptstyle{\scriptscriptstyle -}{\int}$
} \vcenter{\hbox{$\scriptscriptstyle -$
}}\kern-.6\wd0}}{{\setbox0=\hbox{$\scriptscriptstyle{\scriptscriptstyle
-}{\int}$ } \vcenter{\hbox{$\scriptscriptstyle -$ }}\kern-.6\wd0}}\!\int
_{B(x,R)}\left(  T^{2}\left\vert h\right\vert ^{2}+\left\vert H\right\vert
^{2}\right)  dy \label{6.8}%
\end{equation}
where $C>0$ depends only on $d$ and $\alpha$.
\end{proposition}

\begin{proof}
The solvability of (\ref{6.6}) (in the class specified in Proposition
\ref{p4.0}) together with (\ref{6.8}) (for $R=T$) are ensured by \cite[Section
4]{Po-Yu89}; see especially Lemma 4.1 therein. By applying Caccioppoli's
inequality, we obtain (\ref{6.8}) for $R>T$.
\end{proof}

\begin{lemma}
\label{l4.1}Let $h\in B_{\mathcal{A}}^{2}(\mathbb{R}^{d})$ and $H\in
B_{\mathcal{A}}^{2}(\mathbb{R}^{d})^{d}$. Then for any $T>0$, there exists a
unique $u\in B_{\mathcal{A}}^{1,2}(\mathbb{R}^{d})$ such that
\begin{equation}
-\operatorname{div}A\nabla u+T^{-2}u=h+\operatorname{div}H\text{ in
}\mathbb{R}^{d}. \label{6.9}%
\end{equation}
Moreover $u$ solves the equation
\begin{equation}
M(A\nabla u\cdot\nabla\phi+T^{-2}u\phi)=M(h\phi-H\cdot\nabla\phi)\text{ for
all }\phi\in B_{\mathcal{A}}^{1,2}(\mathbb{R}^{d}) \label{6.10}%
\end{equation}
and satisfies the estimate \emph{(\ref{6.8})}.
\end{lemma}

\begin{proof}
See the proof of \cite[Lemma 2.1]{JTW}.
\end{proof}

\begin{remark}
\label{r1.0}\emph{If in (\ref{6.10}) we take }$\phi=1$\emph{ and }$h=0$\emph{,
then we get }$M(v)=0$\emph{, so that the solution of (\ref{6.6}) corresponding
to }$h=0$\emph{, has mean value equal to }$0$\emph{.}
\end{remark}

\begin{proof}
[Proof of Theorem \ref{t4.1}]Theorem \ref{t4.1} has been proved in \cite{JTW}.
Here we repeat the proof for the reader's convenience. 1. \textit{Existence}.
Let us denote by $(u_{T})_{T\geq1}$ the sequence constructed in Lemma
\ref{l4.1} but this time by taking henceforth $h=0$. It satisfies (\ref{6.8}),
that is
\begin{equation}
\sup_{x\in\mathbb{R}^{d},R\geq T}%
\mathchoice {{\setbox0=\hbox{$\displaystyle{\textstyle
-}{\int}$ } \vcenter{\hbox{$\textstyle -$
}}\kern-.6\wd0}}{{\setbox0=\hbox{$\textstyle{\scriptstyle -}{\int}$ } \vcenter{\hbox{$\scriptstyle -$
}}\kern-.6\wd0}}{{\setbox0=\hbox{$\scriptstyle{\scriptscriptstyle -}{\int}$
} \vcenter{\hbox{$\scriptscriptstyle -$
}}\kern-.6\wd0}}{{\setbox0=\hbox{$\scriptscriptstyle{\scriptscriptstyle
-}{\int}$ } \vcenter{\hbox{$\scriptscriptstyle -$ }}\kern-.6\wd0}}\!\int
_{B(x,R)}\left(  T^{-2}\left\vert u_{T}\right\vert ^{2}+\left\vert \nabla
u_{T}\right\vert ^{2}\right)  dy\leq C\sup_{x\in\mathbb{R}^{d}}%
\mathchoice {{\setbox0=\hbox{$\displaystyle{\textstyle
-}{\int}$ } \vcenter{\hbox{$\textstyle -$
}}\kern-.6\wd0}}{{\setbox0=\hbox{$\textstyle{\scriptstyle -}{\int}$ } \vcenter{\hbox{$\scriptstyle -$
}}\kern-.6\wd0}}{{\setbox0=\hbox{$\scriptstyle{\scriptscriptstyle -}{\int}$
} \vcenter{\hbox{$\scriptscriptstyle -$
}}\kern-.6\wd0}}{{\setbox0=\hbox{$\scriptscriptstyle{\scriptscriptstyle
-}{\int}$ } \vcenter{\hbox{$\scriptscriptstyle -$ }}\kern-.6\wd0}}\!\int
_{B(x,T)}\left\vert H\right\vert ^{2}dy\leq C, \label{00}%
\end{equation}
so that the sequence $(\nabla u_{T})_{T\geq1}$ is bounded in $L_{loc}%
^{2}(\mathbb{R}^{d})^{d}$. Therefore, by the weak compactness, $(\nabla
u_{T})_{T\geq1}$ weakly converges in $L_{loc}^{2}(\mathbb{R}^{d})^{d}$ (up to
extraction of a subsequence) to some $V\in L_{loc}^{2}(\mathbb{R}^{d})^{d}$.
From the equality $\partial^{2}u_{T}/\partial y_{i}\partial y_{j}=\partial
^{2}u_{T}/\partial y_{j}\partial y_{i}$, a limit passage in the distributional
sense yields $\partial V_{i}/\partial y_{j}=\partial V_{j}/\partial y_{i}$,
where $V=(V_{j})_{1\leq j\leq d}$. This implies $V=\nabla u$ for some $u\in
H_{loc}^{1}(\mathbb{R}^{d})$. Using the boundedness of $(T^{-1}u_{T})_{T\geq
1}$ in $L_{loc}^{2}(\mathbb{R}^{d})$ (see (\ref{6.8}) once again), we pass to
the limit in the variational formulation of (\ref{6.9}) (as $T\rightarrow
\infty$) to get that $u$ solves (\ref{4.1}). Arguing exactly as in the proof
of (\ref{6.10}) (in Lemma \ref{l4.1}), we show that $u\in H_{loc}%
^{1}(\mathbb{R}^{d})$ also solves the equation
\begin{equation}
M(A\nabla u\cdot\nabla\phi)=M(H\cdot\nabla\phi)\text{ for all }\phi\in
B_{\#\mathcal{A}}^{1,2}(\mathbb{R}^{d}). \label{03}%
\end{equation}
Using the fact that the bilinear form $(u,\phi)\mapsto M(A\nabla u\cdot
\nabla\phi)$ is continuous and coercive on $B_{\#\mathcal{A}}^{1,2}%
(\mathbb{R}^{d})$, we obtain that $u\in B_{\#\mathcal{A}}^{1,2}(\mathbb{R}%
^{d})$ is the unique solution of (\ref{03}). This yields also $V\in
B_{\mathcal{A}}^{2}(\mathbb{R}^{d})^{d}$.

2. \textit{Uniqueness} (of $\nabla u$). Assume that $u\in H_{loc}%
^{1}(\mathbb{R}^{d})$ is such that $-\operatorname{div}A\nabla u=0$ in
$\mathbb{R}^{d}$ and $\nabla u\in B_{\mathcal{A}}^{2}(\mathbb{R}^{d})^{d}$.
Let $\xi\in\mathbb{R}$ be freely fixed. We may replace $u$ by $u-\xi$ and
obtain from Caccioppoli inequality that there is a universal constant
$C_{0}=C_{0}(d,\alpha,\beta)>0$ such that
\[
\int_{B_{r}(x)}\left\vert \nabla u\right\vert ^{2}dy\leq\frac{C_{0}}{r^{2}%
}\int_{B_{2r}(x)\backslash B_{r}(x)}\left\vert u-\xi\right\vert ^{2}dy
\]
where $B_{r}(x)=B(x,r)$. For $\xi
=\mathchoice {{\setbox0=\hbox{$\displaystyle{\textstyle
-}{\int}$ } \vcenter{\hbox{$\textstyle -$
}}\kern-.6\wd0}}{{\setbox0=\hbox{$\textstyle{\scriptstyle -}{\int}$ }
\vcenter{\hbox{$\scriptstyle -$
}}\kern-.6\wd0}}{{\setbox0=\hbox{$\scriptstyle{\scriptscriptstyle -}{\int}$
} \vcenter{\hbox{$\scriptscriptstyle -$
}}\kern-.6\wd0}}{{\setbox0=\hbox{$\scriptscriptstyle{\scriptscriptstyle
-}{\int}$ } \vcenter{\hbox{$\scriptscriptstyle -$ }}\kern-.6\wd0}}\!\int
_{B_{2r}(x)\backslash B_{r}(x)}u$, we derive from Poincar\'{e}'s inequality
that
\[
\int_{B_{r}(x)}\left\vert \nabla u\right\vert ^{2}\leq C_{0}\int
_{B_{2r}(x)\backslash B_{r}(x)}\left\vert \nabla u\right\vert ^{2}.
\]
Now, we make use of the hole-filling technique (see e.g. \cite[p.
82]{Giaquinta}) which consists in adding $C_{0}\int_{B_{r}(x)}\left\vert
\nabla u\right\vert ^{2}$ to both sides of the above inequality to obtain
\[
\int_{B_{r}(x)}\left\vert \nabla u\right\vert ^{2}\leq\frac{C_{0}}{C_{0}%
+1}\int_{B_{2r}(x)}\left\vert \nabla u\right\vert ^{2}.
\]
Now, proceeding as in the proof of (\ref{6.10}), we easily show that
$M(\left\vert \nabla u\right\vert ^{2})=0$, or equivalently $\lim
_{R\rightarrow\infty}R^{-d}\int_{B_{R}(x)}\left\vert \nabla u\right\vert
^{2}=0$. We infer from \cite[page 82]{Giaquinta} that $u$ is constant, so that
$\nabla u=0$. The uniqueness follows, and the proof is completed.
\end{proof}

\section{Homogenization result: Proof of Theorem \ref{t1.1}\label{sec4}}

This section is devoted to the proof of Theorem \ref{t1.1}. Unless otherwise
stated, throughout this section, we do not restrict ourselves to ergodic
algebras with mean value.

\subsection{Passage to the limit}

\begin{proof}
[Proof of Theorem \ref{t1.1}]By (\ref{*11}), the sequence $(u_{\varepsilon
})_{\varepsilon\in E}$ of solutions of problem (\ref{*10}) is bounded in
$H_{0}^{1}(\Omega)$. Then by Theorem \ref{t2.3} there exists a subsequence
$E^{\prime}$ of $E$ and a couple $(u_{0},u_{1})\in H_{0}^{1}(\Omega
;I_{\mathcal{A}}^{2})\times L^{2}(\Omega;B_{\#\mathcal{A}}^{1,2}%
(\mathbb{R}^{d}))$ such that as $E^{\prime}\ni\varepsilon\rightarrow0$,
\begin{align}
u_{\varepsilon}  &  \rightarrow u_{0}\text{ in }L^{2}(\Omega)\text{-weak
}\Sigma\text{,}\label{1.17}\\
\nabla u_{\varepsilon}  &  \rightarrow\nabla u_{0}+\nabla_{y}u_{1}\text{ in
}L^{2}(\Omega)^{d}\text{-weak }\Sigma\text{.} \label{1.18}%
\end{align}
Since $H_{0}^{1}(\Omega)$ is compactly embedded in $L^{2}(\Omega)$, we deduce
from the weak convergence $u_{\varepsilon}\rightarrow\underline{u}_{0}$ in
$H_{0}^{1}(\Omega)$-weak, that $u_{\varepsilon}\rightarrow\underline{u}_{0}$
in $L^{2}(\Omega)$-strong, where $\underline{u}_{0}(x)=M(u_{0}(x,\cdot))$. It
follows that $u_{\varepsilon}\rightarrow\underline{u}_{0}$ in $L^{2}(\Omega
)$-weak $\Sigma$, and hence $u_{0}=\underline{u}_{0}\in H_{0}^{1}(\Omega)$.
Bearing this in mind, let us consider now the following functions: $\psi
_{0}\in\mathcal{C}_{0}^{\infty}(\Omega),$ $\psi_{1}\in\mathcal{C}_{0}^{\infty
}(\Omega)\otimes\mathcal{A}^{\infty}$ and $\Phi=(\psi_{0},\psi_{1})$. Let us
define $\Phi_{\varepsilon}(x)=\psi_{0}(x)+\varepsilon\psi_{1}^{\varepsilon
}(x)$ where $\psi_{1}^{\varepsilon}(x)=\psi_{1}(x,x/\varepsilon)$. Then
$\Phi_{\varepsilon}\in\mathcal{C}_{0}^{\infty}(\Omega)$ and we have
$\nabla\Phi_{\varepsilon}=(\nabla\psi_{0})^{\varepsilon}+\varepsilon
(\nabla\psi_{1})^{\varepsilon}+(\nabla_{y}\psi_{1})^{\varepsilon}$, where we
used the notation $\phi^{\varepsilon}(x)=\phi(x,x/\varepsilon)$. It follows
obviously that as $\varepsilon\rightarrow0$,
\begin{align}
\Phi_{\varepsilon}  &  \rightarrow\psi_{0}\text{ in }L^{2}(\Omega
)\text{-strongly,}\label{1.20}\\
\nabla\Phi_{\varepsilon}  &  \rightarrow\nabla\psi_{0}+\nabla_{y}\psi
_{1}\text{ in }L^{2}(\Omega)^{d}\text{-strongly }\Sigma,\label{1.21}\\
&  (\nabla\Phi_{\varepsilon})_{\varepsilon}\text{ is bounded in }L^{\infty
}(\Omega)^{d}. \label{1.22}%
\end{align}
The variational formulation of (\ref{*10}) (where we have taken $\Phi
_{\varepsilon}$ as test function) reads as
\begin{equation}
\int_{\Omega}(A^{\varepsilon}\nabla u_{\varepsilon}\cdot\nabla\Phi
_{\varepsilon}+V^{\varepsilon}u_{\varepsilon}\cdot\nabla\Phi_{\varepsilon
})dx+\int_{\Omega}(B^{\varepsilon}\nabla u_{\varepsilon}+a_{0}^{\varepsilon
}u_{\varepsilon}+\mu u_{\varepsilon})\Phi_{\varepsilon}dx=\left\langle
\mathbf{f},\Phi_{\varepsilon}\right\rangle \label{1.2222}%
\end{equation}
where $\mathbf{f}=f+\operatorname{div}F\in H^{-1}(\Omega)$. By (\ref{1.18}),
(\ref{1.21}) and (\ref{1.22}), we have that%
\[
\nabla u_{\varepsilon}\cdot\nabla\Phi_{\varepsilon}\rightarrow\mathbb{D}%
\mathbf{u}\cdot\mathbb{D}\Phi\text{ in }L^{2}(\Omega)^{d}\text{-weak }%
\Sigma\text{ as }E^{\prime}\ni\varepsilon\rightarrow0
\]
where we have denoted
\begin{equation}
\mathbb{D}\mathbf{u}=\nabla u_{0}+\nabla_{y}u_{1}\text{ (}\mathbf{u}%
=(u_{0},u_{1})\text{) and}\ \mathbb{D}\Phi=\nabla\psi_{0}+\nabla_{y}\psi_{1}.
\label{1.222}%
\end{equation}
On the other hand, $A\in B_{\mathcal{A}}^{2,\infty}(\mathbb{R}^{d})^{d\times
d}=(B_{\mathcal{A}}^{2}(\mathbb{R}^{d})\cap L^{\infty}(\mathbb{R}%
^{d}))^{d\times d}$; so, taking $A^{\varepsilon}$ as test function in the
definition of weak $\Sigma$-convergence, we get, as $E^{\prime}\ni
\varepsilon\rightarrow0$,
\[
\int_{\Omega}A^{\varepsilon}\nabla u_{\varepsilon}\cdot\nabla\Phi
_{\varepsilon}dx\rightarrow\int_{\Omega}M(A\mathbb{D}u\cdot\mathbb{D}\Phi)dx.
\]
The convergence $u_{\varepsilon}\rightarrow u_{0}$ in $L^{2}(\Omega)$-strong
together with (\ref{1.21}) imply, using $V$ as test function
\[
\int_{\Omega}V^{\varepsilon}u_{\varepsilon}\cdot\nabla\Phi_{\varepsilon
}dx\rightarrow\int_{\Omega}u_{0}M(V\cdot\mathbb{D}\Phi)dx.
\]
Next, we take $B$ as test function and use (\ref{1.18}) and (\ref{1.20}) to
obtain
\[
\int_{\Omega}(B^{\varepsilon}\cdot\nabla u_{\varepsilon})\Phi_{\varepsilon
}dx\rightarrow\int_{\Omega}M(B\cdot\mathbb{D}\mathbf{u})\psi_{0}dx.
\]
By (\ref{1.17}) and (\ref{1.20}) we get, using $a_{0}$ as test function
\[
\int_{\Omega}(a_{0}^{\varepsilon}+\mu)u_{\varepsilon}\Phi_{\varepsilon
}dx\rightarrow\int_{\Omega}(M(a_{0})+\mu)\psi_{0}u_{0}dx.
\]
Finally from the obvious result $\Phi_{\varepsilon}\rightarrow\psi_{0}$ in
$H_{0}^{1}(\Omega)$-weak, we infer $\left\langle \mathbf{f},\Phi_{\varepsilon
}\right\rangle \rightarrow\left\langle \mathbf{f},\psi_{0}\right\rangle $.
Putting together the above convergence results yield in (\ref{1.2222}), as
$E^{\prime}\ni\varepsilon\rightarrow0$,
\begin{equation}
\int_{\Omega}M((A\mathbb{D}\mathbf{u}+Vu_{0})\cdot\mathbb{D}\Phi
)dx+\int_{\Omega}M(B\cdot\mathbb{D}\mathbf{u}+a_{0}u_{0}+\mu u_{0})\psi
_{0}dx=\left\langle \mathbf{f},\psi_{0}\right\rangle \label{1.23}%
\end{equation}
for every $\psi_{0}\in\mathcal{C}_{0}^{\infty}(\Omega)$ and $\psi_{1}%
\in\mathcal{C}_{0}^{\infty}(\Omega)\otimes\mathcal{A}^{\infty}$.

The integral identity (\ref{1.23}) is the so-called \textit{global homogenized
problem}. It is equivalent to the following two identities
\begin{gather}
\int_{\Omega}M(A\mathbb{D}\mathbf{u}+Vu_{0})\cdot\nabla\psi_{0}dx+\int
_{\Omega}M(B\cdot\mathbb{D}\mathbf{u}+a_{0}u_{0}+\mu u_{0})\psi_{0}%
dx=\left\langle \mathbf{f},\psi_{0}\right\rangle \ \forall\psi_{0}%
\in\mathcal{C}_{0}^{\infty}(\Omega),\label{1.24}\\
\int_{\Omega}M((A\mathbb{D}\mathbf{u}+Vu_{0})\cdot\nabla_{y}\psi
_{1})dx=0,\ \ \ \forall\psi_{1}\in\mathcal{C}_{0}^{\infty}(\Omega
)\otimes\mathcal{A}^{\infty}. \label{1.25}%
\end{gather}
Let us derive from the preceding two equations (\ref{1.24}) and (\ref{1.25}),
the homogenized equation whose $u_{0}$ is solution. To this end, let us take
$\psi_{1}(x,y)=\varphi(x)w(y)$, with $\varphi\in\mathcal{C}_{0}^{\infty
}(\Omega)$ and $w\in\mathcal{A}^{\infty}$ in (\ref{1.25}). Then%
\[
\int_{\Omega}M((A\mathbb{D}\mathbf{u}+Vu_{0})\cdot\nabla_{y}w)\varphi
(x)dx=0,\ \ \ \forall\varphi\in\mathcal{C}_{0}^{\infty}(\Omega),\ \forall
w\in\mathcal{A}^{\infty},
\]
which yields%
\begin{equation}
M((A\mathbb{D}\mathbf{u}(x,\cdot)+Vu_{0}(x))\cdot\nabla_{y}w)=0\ \ \ \forall
w\in\mathcal{A}^{\infty},\ \text{a.e. }x\in\Omega. \label{1.26}%
\end{equation}
So, fix $\xi\in\mathbb{R}^{d}$, $\lambda\in\mathbb{R}$ and consider the
equation
\begin{equation}
-\operatorname{div}_{y}(A(\xi+\nabla_{y}v_{\xi,\lambda})+V\lambda)=0\text{ in
}\mathbb{R}^{d},\ v_{\xi,\lambda}\in B_{\#\mathcal{A}}^{1,2}(\mathbb{R}^{d})
\label{1.28}%
\end{equation}
that is a linear and a well-posed problem.

We seek for solutions of the form $v_{\xi,\lambda}=u_{\xi}+\eta_{\lambda},$
where $u_{\xi}$ and $\eta_{\lambda}$ are solutions, respectively, of the
following equations:
\begin{align}
-\operatorname{div}_{y}(A(\xi+\nabla_{y}u_{\xi}))  &  =0\text{ in}%
\ \mathbb{R}^{d}\text{, }\ u_{\xi}\in B_{\#\mathcal{A}}^{1,2}(\mathbb{R}%
^{d}),\label{1.29}\\
-\operatorname{div}_{y}(A\nabla_{y}\eta_{\lambda})  &  =\operatorname{div}%
_{y}(V\lambda)\text{ in }\mathbb{R}^{d}\text{,}\ \ \ \eta_{\lambda}\in
B_{\#\mathcal{A}}^{1,2}(\mathbb{R}^{d}). \label{1.30}%
\end{align}
Equation (\ref{1.29}) can be rewritten as
\[
-\operatorname{div}_{y}(A\nabla_{y}u_{\xi})=\operatorname{div}_{y}(A\xi).
\]
In view of Assumptions (H1), (H2) and (H3), we get by Theorem \ref{t4.1} the
existence of unique solutions to (\ref{1.28}), (\ref{1.29}) and (\ref{1.30}).

Going back to (\ref{1.28}) and taking $\xi=\nabla u_{0}(x)$ and $\lambda
=u_{0}(x)$, we get that $u_{\nabla u_{0}(x),u_{0}(x)}$ is the unique solution
(in the sense of Theorem \ref{t4.1}) of
\begin{equation}
-\operatorname{div}_{y}(A(\nabla u_{0}(x)+\nabla_{y}u_{\nabla u_{0}%
(x),u_{0}(x)})+Vu_{0}(x))=0\text{ \ in}\ \mathbb{R}^{d}. \label{1.31}%
\end{equation}
Comparing the variational form (with respect to the duality arising from the
mean value as in (\ref{1.26})) of (\ref{1.31}) with (\ref{1.26}), we deduce
from the uniqueness of the solution in $B_{\#\mathcal{A}}^{1,2}(\mathbb{R}%
^{d})$ that
\begin{equation}
u_{1}(x,\cdot)=u_{\nabla u_{0}(x),u_{0}(x)}=u_{\nabla u_{0}(x)}+\eta
_{u_{0}(x)}\text{ in }B_{\#\mathcal{A}}^{1,2}(\mathbb{R}^{d})\text{ \ a.e.
}x\in\Omega. \label{1.32}%
\end{equation}
Denoting $\chi=(\chi_{j})_{1\leq j\leq d},$ where for $1\leq j\leq d$,
$\chi_{j}=u_{e_{j}}$ is the solution of (\ref{1.29}) determined by Theorem
\ref{t4.1}, and corresponding to $\xi=e_{j}$ ($e_{j}$ the $j$-th vector of the
canonical basis of $\mathbb{R}^{d}$), we deduce from the uniqueness of
$u_{\nabla u_{0}}$ that $u_{\nabla u_{0}}=\chi\cdot\nabla u_{0}$. Also we have
in (\ref{1.30}) $\eta_{u_{0}}=\chi_{0}u_{0}$ where $\chi_{0}=\eta_{1}$ so that
(\ref{1.24}) can be written as
\[%
\begin{array}
[c]{l}%
{\displaystyle\int_{\Omega}}
M(A(I_{d}+\nabla_{y}\chi)\cdot\nabla u_{0}+(A\nabla_{y}\chi_{0}+V\mathbf{)}%
u_{0})\cdot\nabla\psi_{0}dx+\\
\ \ \ +%
{\displaystyle\int_{\Omega}}
[(M(B(I_{d}+\nabla_{y}\chi))\cdot\nabla u_{0})\psi_{0}+(M(B\cdot\nabla\chi
_{0}+a_{0})+\mu)u_{0}\psi_{0}]dx=\left\langle \mathbf{f},\psi_{0}\right\rangle
.
\end{array}
\]
Setting $\widehat{A}=M(A(I_{d}+\nabla_{y}\chi))$, $\widehat{V}=M(A\nabla
_{y}\chi_{0}+V\mathbf{)}$, $\widehat{B}=M(B(I_{d}+\nabla_{y}\chi))$ and
$\widehat{a}_{0}=M(B\cdot\nabla\chi_{0}+a_{0})$, we end up with
\[
\int_{\Omega}(\widehat{A}\nabla u_{0}+\widehat{V}u_{0})\cdot\nabla\psi
_{0}dx+\int_{\Omega}(\widehat{B}\cdot\nabla u_{0}+\widehat{a}_{0}u_{0}+\mu
u_{0})\psi_{0}dx=\left\langle \mathbf{f},\psi_{0}\right\rangle \ \ \ \forall
\psi_{0}\in\mathcal{C}_{0}^{\infty}(\Omega),
\]
which is nothing else but the variational formulation of the homogenized
equation, viz.
\begin{equation}
\mathcal{P}_{0}u_{0}=f+\operatorname{div}F\text{ \ in }\Omega,\ \ \ u_{0}\in
H_{0}^{1}(\Omega). \label{1.33}%
\end{equation}
where
\begin{equation}
\mathcal{P}_{0}=-\operatorname{div}(\widehat{A}\nabla+\widehat{V})+\widehat
{B}\nabla+(\widehat{a}_{0}+\mu). \label{1.33'}%
\end{equation}
Moreover, the local corrector problems are given by:
\begin{align*}
-\operatorname{div}_{y}(A(e_{j}+\nabla_{y}\chi_{j}))  &  =0\text{ in
}\mathbb{R}^{d},\ \ \ \chi_{j}\in B_{\#\mathcal{A}}^{1,2}(\mathbb{R}%
^{d}),\ j=1,...,d,\\
-\operatorname{div}_{y}(A\nabla_{y}\chi_{0}+V)  &  =0\text{ in }\mathbb{R}%
^{d},\ \ \ \chi_{0}\in B_{\#\mathcal{A}}^{1,2}(\mathbb{R}^{d}).
\end{align*}
We may easily show that the matrix $\widehat{A}$ is symmetric and coercive.
Also $\widehat{B},\widehat{V}\in\mathbb{R}^{d}$ and $\widehat{a}_{0}%
\in\mathbb{R}$. Hence the uniqueness of $u_{0}$ solution to (\ref{1.33}) is
ensured. The convergence of the entire sequence $(u_{\varepsilon
})_{\varepsilon>0}$ in $H_{0}^{1}(\Omega)$ therefore stems from the uniqueness
of the solution to the homogenized equation (\ref{1.33}).

We recall that one has
\begin{align*}
u_{\varepsilon}  &  \rightarrow u_{0}\text{ in }H_{0}^{1}(\Omega)\text{-weak
and in }L^{2}(\Omega)\text{-strong}\\
\nabla u_{\varepsilon}  &  \rightarrow\nabla u_{0}+\nabla_{y}u_{1}\text{ in
}L^{2}(\Omega)^{d}\text{-weak }\Sigma.
\end{align*}
This being so, let us prove the convergence (\ref{1.7'}). Let $r_{\varepsilon
}=u_{\varepsilon}-u_{0}-\varepsilon u_{1}^{\varepsilon}$ where
\[
u_{1}^{\varepsilon}(x)=u_{1}(x,x/\varepsilon)=\chi(x/\varepsilon)\nabla
u_{0}(x)+\chi_{0}(x/\varepsilon)u_{0}(x);
\]
then $\nabla r_{\varepsilon}=\nabla u_{\varepsilon}-\nabla u_{0}-(\nabla
_{y}u_{1})^{\varepsilon}-\varepsilon(\nabla u_{1})^{\varepsilon}$. We recall
that we have assumed that $u_{1}\in H^{1}(\Omega;\mathcal{A}^{1})$, so that we
may take $u_{1}$, $\nabla u_{1}$ and $\nabla_{y}u_{1}$ as test functions since
they belong to $L^{2}(\Omega;\mathcal{A})$. We therefore have, according to
(\ref{1.2}) and (\ref{1.3}) and using the variational formulation of
(\ref{*10}) with a test function $u_{\varepsilon}$:
\begin{align}
0  &  \leq\alpha||\nabla r_{\varepsilon}||_{L^{2}(\Omega)}^{2}\leq\int
_{\Omega}A^{\varepsilon}\nabla r_{\varepsilon}\cdot\nabla r_{\varepsilon
}dx\label{1*}\\
&  =\int_{\Omega}A^{\varepsilon}\nabla u_{\varepsilon}\cdot\nabla
u_{\varepsilon}-\int_{\Omega}A^{\varepsilon}\nabla u_{\varepsilon}\cdot(\nabla
u_{0}+(\nabla_{y}u_{1})^{\varepsilon}+\varepsilon(\nabla u_{1})^{\varepsilon
})\nonumber\\
&  -\int_{\Omega}A^{\varepsilon}(\nabla u_{0}+(\nabla_{y}u_{1})^{\varepsilon
}+\varepsilon(\nabla u_{1})^{\varepsilon})\cdot\nabla u_{\varepsilon
}\nonumber\\
&  +\int_{\Omega}A^{\varepsilon}(\nabla u_{0}+(\nabla_{y}u_{1})^{\varepsilon
}+\varepsilon(\nabla u_{1})^{\varepsilon})\cdot(\nabla u_{0}+(\nabla_{y}%
u_{1})^{\varepsilon}+\varepsilon(\nabla u_{1})^{\varepsilon})\nonumber\\
&  =\left\langle \mathbf{f},u_{\varepsilon}\right\rangle -\int_{\Omega
}(u_{\varepsilon}V^{\varepsilon}\cdot\nabla u_{\varepsilon}-(B^{\varepsilon
}\nabla u_{\varepsilon})u_{\varepsilon}-(a_{0}^{\varepsilon}+\mu)\left\vert
u_{\varepsilon}\right\vert ^{2})\nonumber\\
&  +\int_{\Omega}A^{\varepsilon}(\nabla u_{0}+(\nabla_{y}u_{1})^{\varepsilon
}+\varepsilon(\nabla u_{1})^{\varepsilon})\cdot(\nabla u_{0}+(\nabla_{y}%
u_{1})^{\varepsilon}+\varepsilon(\nabla u_{1})^{\varepsilon})\nonumber\\
&  -\int_{\Omega}A^{\varepsilon}\nabla u_{\varepsilon}\cdot(\nabla
u_{0}+(\nabla_{y}u_{1})^{\varepsilon}+\varepsilon(\nabla u_{1})^{\varepsilon
})-\int_{\Omega}A^{\varepsilon}(\nabla u_{0}+(\nabla_{y}u_{1})^{\varepsilon
}+\varepsilon(\nabla u_{1})^{\varepsilon})\cdot\nabla u_{\varepsilon
}\nonumber\\
&  +\varepsilon\int_{\Omega}A^{\varepsilon}(\nabla u_{0}+(\nabla_{y}%
u_{1})^{\varepsilon}+\varepsilon(\nabla u_{1})^{\varepsilon})\cdot(\nabla
u_{1})^{\varepsilon}-\varepsilon\int_{\Omega}(A^{\varepsilon}\nabla
u_{\varepsilon}\cdot(\nabla u_{1})^{\varepsilon}+(A\nabla u_{1})^{\varepsilon
}\cdot\nabla u_{\varepsilon}).\nonumber
\end{align}
The only term in (\ref{1*}) in which the limit passage requires explanation
is
\[
\int_{\Omega}A^{\varepsilon}\nabla u_{\varepsilon}\cdot(\nabla u_{0}%
+(\nabla_{y}u_{1})^{\varepsilon}+\varepsilon(\nabla u_{1})^{\varepsilon})dx.
\]
In the above mentioned term, we have that
\[
\nabla u_{\varepsilon}\rightarrow\nabla u_{0}+\nabla_{y}u_{1}\text{ in }%
L^{2}(\Omega)^{d}\text{-weak }\Sigma.
\]
Now since $\nabla_{y}u_{1}$ and $\nabla u_{1}$ both belong to $L^{2}%
(\Omega;\mathcal{A})^{d}$, we get
\[
\nabla u_{0}+(\nabla_{y}u_{1})^{\varepsilon}+\varepsilon(\nabla u_{1}%
)^{\varepsilon}\rightarrow\nabla u_{0}+\nabla_{y}u_{1}\text{ in }L^{2}%
(\Omega)^{d}\text{-strong }\Sigma.
\]
Hence, taking $A^{\varepsilon}(x)=A(x/\varepsilon)$ as test function (recall
that $A\in(B_{\mathcal{A}}^{2}(\mathbb{R}^{d})\cap L^{\infty}(\mathbb{R}%
^{d}))^{d\times d}$) and letting $\varepsilon\rightarrow0$ we get readily
\[
\int_{\Omega}A^{\varepsilon}\nabla u_{\varepsilon}\cdot(\nabla u_{0}%
+(\nabla_{y}u_{1})^{\varepsilon}+\varepsilon(\nabla u_{1})^{\varepsilon
})dx\rightarrow\int_{\Omega}M(A\mathbb{D}\mathbf{u}\cdot\mathbb{D}%
\mathbf{u})dx.
\]
Therefore, passing to the limit as $\varepsilon\rightarrow0$ and taking into
account (\ref{1.2222}), the right-hand side tends to%
\[
\left\langle \mathbf{f},u_{0}\right\rangle -\int_{\Omega}M(Vu_{0}%
\mathbb{D}\mathbf{u})dx-\int_{\Omega}M(A\mathbb{D}\mathbf{u}\cdot
\mathbb{D}\mathbf{u})dx-\int_{\Omega}M(B\cdot\mathbb{D}\mathbf{u}+a_{0}%
u_{0}+\mu u_{0})u_{0}dx
\]
that is zero by (\ref{1.23}) and density arguments. Then, as $\varepsilon
\rightarrow0$, $\left\Vert \nabla r_{\varepsilon}\right\Vert _{L^{2}(\Omega
)}\rightarrow0$, which proves the convergence (\ref{1.7'}). This ends the
proof of Theorem \ref{t1.1}.
\end{proof}

\subsection{Some applications of Theorem \ref{t1.1}}

We provide here some concrete situations for which Theorem \ref{t1.1} holds
true. The framework is in accordance with assumption (H3).

\subsubsection{\textbf{Problem 1 (Periodic homogenization)}}

\begin{itemize}
\item[(H3)$_{1}$] We assume that the functions $A$, $V$ and $a_{0}$ are
periodic of period $1$ in each coordinate.
\end{itemize}

Then, this suggests us to take $\mathcal{A}=\mathcal{C}_{per}(Y)$, the Banach
algebra of continuous $Y$-periodic functions defined on $\mathbb{R}^{d}$,
where $Y=(0,1)^{d}$. Therefore we obtain (H3) with $\mathcal{A}=\mathcal{C}%
_{per}(Y)$. In this case we have $B_{\mathcal{A}}^{p}(\mathbb{R}^{d}%
)=L_{per}^{p}(Y)$ for $1\leq p\leq\infty$ and $B_{\mathcal{A}}^{1,2}%
(\mathbb{R}^{d})=H_{per}^{1}(Y)$.

In this special case, the homogenization result reads as follows.

\begin{theorem}
\label{t6.1}Assume \emph{(H1)-(H2)} and \emph{(H3)}$_{1}$ hold. Then the
sequence of solutions of \emph{(\ref{1.1})} weakly converges in $H_{0}%
^{1}(\Omega)$ to the solution of \emph{(\ref{1.4'})} where
\begin{align*}
\widehat{A}  &  =\int_{Y}A(I_{d}+\nabla_{y}\chi)dy,\ \widehat{B}=\int
_{Y}B(I_{d}+\nabla_{y}\chi)dy,\ \widehat{V}=\int_{Y}(V+A\nabla_{y}\chi
_{0})dy\text{ and }\\
\widehat{a}_{0}  &  =\int_{Y}(B\cdot\nabla_{y}\chi_{0}+a_{0})dy,
\end{align*}
$\chi=(\chi_{j})_{1\leq j\leq d}\in(H_{per}^{1}(Y)/\mathbb{R})^{d}$ and
$\chi_{0}\in H_{per}^{1}(Y)/\mathbb{R}$ being defined by the cell problems
\[
-\operatorname{div}_{y}(A(e_{j}+\nabla_{y}\chi_{j}))=0\text{ in }Y\text{ and
}-\operatorname{div}_{y}(A\nabla_{y}\chi_{0}+V)=0\text{ in }Y.
\]

\end{theorem}

\begin{proof}
The above result stems from the characterization of the mean value in the
periodic setting: for $u\in L_{per}^{p}(Y)$, $M(u)=\int_{Y}u(y)dy$.
\end{proof}

\subsubsection{\textbf{Problem 2 (Almost periodic homogenization)}}

The functions $A$, $V$ and $a_{0}$ are assumed to be Besicovitch almost
periodic \cite{Besicovitch}. We then get (H3) with $\mathcal{A}=AP(\mathbb{R}%
^{d})$, where $AP(\mathbb{R}^{d})$ \cite{Besicovitch, Bohr} is the space of
continuous almost periodic functions on $\mathbb{R}^{d}$. In this case, the
mean value of a function $u\in AP(\mathbb{R}^{d})$ can be obtained as the
unique constant belonging to the closed convex hull of the family of the
translates $(u(\cdot+a))_{a\in\mathbb{R}^{d}}$; see e.g. \cite{Jacobs}.

\subsubsection{\textbf{Problem 3 (Asymptotic periodic homogenization)}}

Let $\mathcal{B}_{\infty}(\mathbb{R}^{d})$ denote the space of all continuous
functions $\psi\in\mathcal{C}(\mathbb{R}^{d})$ such that $\psi(\zeta)$ has a
finite limit in $\mathbb{R}$ as $\left\vert \zeta\right\vert \rightarrow
\infty$. It is known that $\mathcal{B}_{\infty}(\mathbb{R}^{d})$ is an algebra
with mean value on $\mathbb{R}^{d}$ for which the mean value of a function
$u\in\mathcal{B}_{\infty}(\mathbb{R}^{d})$ is obtained as the limit at infinity.

This being so, we mean to homogenize problem (\ref{1.1}) under the assumption

\begin{itemize}
\item[(H3)$_{2}$] $A\in(L_{per}^{2}(Y))^{d\times d}$, $V,B\in\mathcal{B}%
_{\infty}(\mathbb{R}^{d})^{d}$ and $a_{0}\in L_{per}^{2}(Y)$.
\end{itemize}

It is an easy exercise to see that the appropriate algebra with mean value
here is $\mathcal{A}=\mathcal{B}_{\infty}(\mathbb{R}^{d})+\mathcal{C}%
_{per}(Y)\equiv\mathcal{B}_{\infty,per}(\mathbb{R}^{d})$. Indeed, we know that
$\mathcal{B}_{\infty,per}(\mathbb{R}^{d})$ is an algebra with mean value on
$\mathbb{R}^{d}$ with the property that $\mathcal{B}_{\infty,per}%
(\mathbb{R}^{d})=\mathcal{C}_{0}(\mathbb{R}^{d})\oplus\mathcal{C}_{per}(Y)$
(direct and topological sum; see e.g. \cite{NA}) where $\mathcal{C}%
_{0}(\mathbb{R}^{d})$ stands for the space of those continuous functions $u$
in $\mathbb{R}^{d}$ that vanish at infinity. Since $\lim_{\left\vert
y\right\vert \rightarrow\infty}u(y)=0$ for any $u\in\mathcal{C}_{0}%
(\mathbb{R}^{d})$, any element in $\mathcal{B}_{\infty,per}(\mathbb{R}^{d})$
is asymptotically a periodic function.

The general asymptotic periodic homogenization is obtained by assuming that

\begin{itemize}
\item[(H3)$_{3}$] $A,B,V$ and $a_{0}$ belong to $B_{\mathcal{A}}%
^{2}(\mathbb{R}^{d})^{d\times d}$, $B_{\mathcal{A}}^{2}(\mathbb{R}^{d})^{d}$,
$B_{\mathcal{A}}^{2}(\mathbb{R}^{d})^{d}$ and $B_{\mathcal{A}}^{2}%
(\mathbb{R}^{d})$, respectively, with $\mathcal{A}=\mathcal{B}_{\infty
,per}(\mathbb{R}^{d})$.
\end{itemize}

\noindent Then we get the homogenization of (\ref{1.1}) under either
(H3)$_{2}$ or (H3)$_{3}$.

\subsubsection{\textbf{Problem 4 (Asymptotic almost periodic homogenization)}}

In Problem 3 above, we may replace $\mathcal{C}_{per}(Y)$ by $AP(\mathbb{R}%
^{d})$ and get the algebra of asymptotic almost periodic functions denoted by
$\mathcal{B}_{\infty,AP}(\mathbb{R}^{d})=\mathcal{B}_{\infty}(\mathbb{R}%
^{d})+AP(\mathbb{R}^{d})$. In this case, we may assume that the functions $A$,
$V$ and $a_{0}$ belong to $B_{\mathcal{A}}^{2}(\mathbb{R}^{d})^{d\times d}$,
$B_{\mathcal{A}}^{2}(\mathbb{R}^{d})^{d}$, $B_{\mathcal{A}}^{2}(\mathbb{R}%
^{d})^{d}$ and $B_{\mathcal{A}}^{2}(\mathbb{R}^{d})$, respectively, where
$\mathcal{A}=\mathcal{B}_{\infty,AP}(\mathbb{R}^{d})$. The conclusion of
Theorem \ref{t1.1} holds as well.

\section{Existence of true asymptotic periodic correctors: Proof of Theorem
\ref{t1.2}\label{sec5'}}

We consider, in the asymptotic periodic setting, the corrector problems
(\ref{1.29}) and (\ref{1.30}) in which we choose $\xi=e_{j}$ (the $j$th vector
of the canonical basis in $\mathbb{R}^{d}$) and $\lambda=1$. The main aim of
this section is to derive the existence of a true corrector solution of either
equation (\ref{2.0}) or (\ref{2.0'}) below:
\begin{equation}
-\operatorname{div}_{y}(A\nabla_{y}\chi_{0}+V)=0\text{ in }\mathbb{R}^{d}
\label{2.0}%
\end{equation}%
\begin{equation}
-\operatorname{div}_{y}(A(e_{j}+\nabla_{y}\chi_{j}))=0\text{ in }%
\mathbb{R}^{d}\text{ }(1\leq j\leq d). \label{2.0'}%
\end{equation}
We investigate the properties of the solution of problem (\ref{2.0}). The
obtained properties will also hold for equation (\ref{2.0'}) since
(\ref{2.0'}) can be obtained from (\ref{2.0}) by replacing in (\ref{2.0}) the
function $V$ by $Ae_{j}$.

\subsection{Preliminaries}

In this section, we are dealing with the special case of asymptotic periodic
functions. The algebra with mean value associated is therefore $\mathcal{A}%
=\mathcal{B}_{\infty,per}(\mathbb{R}^{d}):=\mathcal{C}_{0}(\mathbb{R}%
^{d})\oplus\mathcal{C}_{per}(Y)$. Let us first recall that the Besicovitch
space associated to $\mathcal{A}$ is $B_{\mathcal{A}}^{2}(\mathbb{R}%
^{d})\equiv L_{\infty,per}^{2}(\mathbb{R}^{d})$. Endowed with the seminorm
$\left\Vert u\right\Vert _{2}=(M(\left\vert u\right\vert ^{2}))^{1/2}$ ($u\in
L_{\infty,per}^{2}(\mathbb{R}^{d})$), $L_{\infty,per}^{2}(\mathbb{R}^{d})$ is
a Fr\'{e}chet space. We also recall that $\mathcal{L}_{\infty,per}%
^{2}(\mathbb{R}^{d})=L_{\infty,per}^{2}(\mathbb{R}^{d})/\mathcal{N}$ (where
$\mathcal{N}=\{u\in L_{\infty,per}^{2}(\mathbb{R}^{d}):\left\Vert u\right\Vert
_{2}=0\}$) is a Hilbert space. In view of the decomposition $L_{\infty
,per}^{2}(\mathbb{R}^{d})=L_{0}^{2}(\mathbb{R}^{d})+L_{per}^{2}(Y)$ (where
$L_{0}^{2}(\mathbb{R}^{d})$ is the closure in $\mathfrak{M}^{2}(\mathbb{R}%
^{d})$ of $\mathcal{C}_{0}(\mathbb{R}^{d})$ with respect to the seminorm
$\left\Vert \cdot\right\Vert _{2}$), we notice that those functions in
$L_{\infty,per}^{2}(\mathbb{R}^{d})$ with mean value equal to zero, are those
for which the periodic component has mean value equal to zero since $M(v)=0$
for all $v\in L_{0}^{2}(\mathbb{R}^{d})$. It is an easy task to see that
$\mathcal{N}=L_{0}^{2}(\mathbb{R}^{d})$, so that $\mathcal{L}_{\infty,per}%
^{2}(\mathbb{R}^{d})=L_{\infty,per}^{2}(\mathbb{R}^{d})/L_{0}^{2}%
(\mathbb{R}^{d})$. By the first isomorphism theorem, we have that
$\mathcal{L}_{\infty,per}^{2}(\mathbb{R}^{d})\cong L_{per}^{2}(Y)$. Denote by
$\ell$ the above isomorphism and by $\varrho$ the canonical mapping from
$L_{\infty,per}^{2}(\mathbb{R}^{d})$ onto $\mathcal{L}_{\infty,per}%
^{2}(\mathbb{R}^{d})$: $\varrho(u)=u+L_{0}^{2}(\mathbb{R}^{d})\equiv
\overset{\circ}{u}$. We still denote by $M$ the mean value on $\mathcal{L}%
_{\infty,per}^{2}(\mathbb{R}^{d})$, which is well defined since for $v\in
L_{0}^{2}(\mathbb{R}^{d})$ one has $M(v)=0$. It is easy to see that $\ell$ is
an isometric (topological) isomorphism which is defined by
\begin{equation}
\ell(u_{per}+L_{0}^{2}(\mathbb{R}^{d}))=u_{per}\text{ for }u_{per}\in
L_{per}^{2}(Y). \label{8.1}%
\end{equation}
Also the following identities are easily verified:
\begin{equation}
M(\ell(\overset{\circ}{u}_{per}))=M(u_{per})\text{ for }\overset{\circ}%
{u}_{per}=\varrho(u_{per})=u_{per}+L_{0}^{2}(\mathbb{R}^{d})\in\mathcal{L}%
_{\infty,per}^{2}(\mathbb{R}^{d}) \label{8.2}%
\end{equation}%
\begin{equation}
\ell(\overset{\circ}{u}_{per}\overset{\circ}{v}_{per})=\ell(\overset{\circ}%
{u}_{per})\ell(\overset{\circ}{v}_{per})\text{ for }u_{per}\in L_{per}%
^{2}(Y)\text{ and }v_{per}\in L_{per}^{2}(Y)\cap L^{\infty}(\mathbb{R}^{d}).
\label{8.3}%
\end{equation}
Let $\mathcal{H}_{\infty,per}^{1}(\mathbb{R}^{d})=\ell^{-1}(H_{per}^{1}(Y))$.
Then proceeding as in \cite{CMP}, we may define (for each $1\leq i\leq d$) a
linear operator $\overline{\partial}/\partial y_{i}:\mathcal{H}_{\infty
,per}^{1}(\mathbb{R}^{d})\rightarrow\mathcal{L}_{\infty,per}^{2}%
(\mathbb{R}^{d})$ by (see (\ref{0.3}))
\begin{equation}
\frac{\overline{\partial}\varrho(u_{per})}{\partial y_{i}}=\varrho\left(
\frac{\partial u_{per}}{\partial y_{i}}\right)  \text{ for }u_{per}\in
H_{per}^{1}(Y). \label{8.4}%
\end{equation}
This leads to a precise definition of $\mathcal{H}_{\infty,per}^{1}%
(\mathbb{R}^{d})$:
\[
\mathcal{H}_{\infty,per}^{1}(\mathbb{R}^{d})=\{\overset{\circ}{u}_{per}%
\in\mathcal{L}_{\infty,per}^{2}(\mathbb{R}^{d}):\overline{\nabla}%
\overset{\circ}{u}_{per}\in\mathcal{L}_{\infty,per}^{2}(\mathbb{R}^{d})^{d}\}
\]
where $\overline{\nabla}=(\overline{\partial}/\partial y_{i})_{1\leq i\leq d}%
$. From the definition of $\ell$ we see that
\begin{equation}
\nabla\ell(\overset{\circ}{u}_{per})=\ell(\overline{\nabla}\overset{\circ}%
{u}_{per})=\ell(\varrho(\nabla u_{per}))\text{, i.e. }\nabla\ell(u_{per}%
+L_{0}^{2}(\mathbb{R}^{d}))=\ell(\nabla u_{per}+L_{0}^{2}(\mathbb{R}^{d})).
\label{8.5}%
\end{equation}
From the above equality (\ref{8.5}) we infer that $\mathcal{H}_{\infty
,per}^{1}(\mathbb{R}^{d})=\varrho(H_{\infty,per}^{1}(\mathbb{R}^{d}))$ where
\begin{equation}
H_{\infty,per}^{1}(\mathbb{R}^{d})=\{u\in L_{\infty,per}^{2}(\mathbb{R}%
^{d}):\nabla u\in L_{\infty,per}^{2}(\mathbb{R}^{d})^{d}\}. \label{2.00}%
\end{equation}
We close these preliminaries by defining a suitable subspace of $\mathcal{H}%
_{\infty,per}^{1}(\mathbb{R}^{d})$ to which the solution of the corrector
problem will belong to. Let $\mathcal{H}_{\#}^{1}(\mathbb{R}^{d}%
)=\{\overset{\circ}{u}_{per}\in\mathcal{H}_{\infty,per}^{1}(\mathbb{R}%
^{d}):M(\overset{\circ}{u}_{per})=0\}\equiv\ell^{-1}(H_{per}^{1}%
(Y)/\mathbb{R})$ where $H_{per}^{1}(Y)/\mathbb{R}=\{u\in H_{per}^{1}%
(Y):\int_{Y}u=0\}$. We endow $\mathcal{H}_{\#}^{1}(\mathbb{R}^{d})$ with the
gradient seminorm denoted by $\left\Vert \cdot\right\Vert _{\#}$:
\[
\left\Vert \overset{\circ}{u}_{per}\right\Vert _{\#}=\left\Vert \overline
{\nabla}\overset{\circ}{u}_{per}\right\Vert _{\mathcal{L}_{\infty,per}%
^{2}(\mathbb{R}^{d})}\equiv\left[  M\left(  \left\vert \overline{\nabla
}\overset{\circ}{u}_{per}\right\vert ^{2}\right)  \right]  ^{\frac{1}{2}%
}\text{ for }\overset{\circ}{u}_{per}\in\mathcal{H}_{\#}^{1}(\mathbb{R}^{d}).
\]
Then since $\ell$ is an isometry, we make use of (\ref{8.5}) to see that
\begin{equation}
\left\Vert \overset{\circ}{u}_{per}\right\Vert _{\#}=\left\Vert \nabla
u_{per}\right\Vert _{L^{2}(Y)},\ \ u_{per}\in H_{per}^{1}(Y)/\mathbb{R}.
\label{8.6}%
\end{equation}
It comes straightforwardly from (\ref{8.6}) that $\left\Vert \cdot\right\Vert
_{\#}$ is actually a norm on $\mathcal{H}_{\#}^{1}(\mathbb{R}^{d})$ which
makes it a Hilbert space.

Finally, we note that $L^{p}(\mathbb{R}^{d})\subset L_{0}^{2}(\mathbb{R}^{d})$
for any $2\leq p<\infty$, so the results that will be proved below will in
particular hold for the corresponding occurrences taken in $L^{p}%
(\mathbb{R}^{d})+L_{per}^{2}(Y)$ ($2\leq p<\infty$).

\subsection{Proof of Theorem \ref{t1.2}}

\begin{proof}
[Proof of Theorem \ref{t1.2}]Since $A=A_{0}+A_{per}$ and $V=V_{0}+V_{per}$
with $A_{0}$, $V_{0}$ having entries in $L_{0}^{2}(\mathbb{R}^{d})$ and
$A_{per}$, $V_{per}$ having entries in $L_{per}^{2}(Y)$, then $A_{per}%
=\ell(\varrho(A))$ and $V_{per}=\ell(\varrho(V))$. Consider the periodic
corrector problem
\begin{equation}
-\operatorname{div}(A_{per}\nabla\chi_{per}+V_{per})=0\text{ in }%
\mathbb{R}^{d},\ \int_{Y}\chi_{per}=0\text{.} \label{8.8}%
\end{equation}
In view of (\ref{2.2}), problem (\ref{8.8}) possesses a unique solution
$\chi_{per}\in H_{per}^{1}(Y)/\mathbb{R}$. Set $u=\ell^{-1}(\chi_{per}%
)\in\mathcal{H}_{\#}^{1}(\mathbb{R}^{d})$; then $u=\overset{\circ}{\chi}%
_{per}$ and $\ell(u)=\chi_{per}$. Next, let $\phi\in\mathcal{A}^{\infty}$ be
arbitrarily fixed, $\mathcal{A}=\mathcal{B}_{\infty,per}(\mathbb{R}^{d})$.
Then $\ell(\varrho(\phi))\in\mathcal{C}_{per}^{\infty}(Y)$ and the variational
formulation of (\ref{8.8}) gives
\begin{equation}
M\left(  (A_{per}\nabla\chi_{per}+V_{per})\cdot\nabla\ell(\varrho
(\phi))\right)  =0. \label{8.9}%
\end{equation}
But
\begin{align*}
A_{per}\nabla\chi_{per}  &  =\ell(\varrho(A))\ell(\varrho(\nabla\chi
_{per}))=\ell(\varrho(A))\ell(\overline{\nabla}\varrho(\chi_{per}))\\
&  =\ell(\overset{\circ}{A}\overline{\nabla}\overset{\circ}{\chi}_{per})\text{
thanks to (\ref{8.3}).}%
\end{align*}
The equation (\ref{8.9}) is therefore equivalent to
\[
M[\ell((\overset{\circ}{A}\overline{\nabla}\overset{\circ}{\chi}%
_{per}+\overset{\circ}{V})\cdot\overline{\nabla}\overset{\circ}{\phi
})]=0\text{ for all }\phi\in\mathcal{A}^{\infty},
\]
or, in view of (\ref{8.2}),
\begin{equation}
M[(\overset{\circ}{A}\overline{\nabla}\overset{\circ}{\chi}_{per}%
+\overset{\circ}{V})\cdot\overline{\nabla}\overset{\circ}{\phi}]=0\text{ for
all }\phi\in\mathcal{A}^{\infty}. \label{8.10}%
\end{equation}
Owing to the properties of the matrix $A$, we see that $\overset{\circ}{\chi
}_{per}\in\mathcal{H}_{\#}^{1}(\mathbb{R}^{d})$ is the unique solution to
(\ref{8.10}). However, with the same notations as above, if we proceed exactly
as we did in order to obtain (\ref{6.10}), then one easily shows that the
solution $\chi_{0}$ of (\ref{2.0}) provided by Theorem \ref{t4.1} satisfies
the variational equation (\ref{8.10}), so that by the uniqueness of the
solution to (\ref{8.10}), we get $\overset{\circ}{\chi}_{0}=\overset{\circ
}{\chi}_{per}$. It emerges the existence of (a non unique) $\chi_{0}^{0}\in
L_{0}^{2}(\mathbb{R}^{d})$ such that $\chi_{0}=\chi_{0}^{0}+\chi_{per}$. This
concludes the proof of Theorem \ref{t1.2}.
\end{proof}

\section{Convergence rates: Proof of Theorem \ref{t2.1}\label{sec5}}

\begin{lemma}
\label{l2.4}Let $g\in L_{\infty,per}^{2}(\mathbb{R}^{d})$ be such that
$M(g)=0$. Then there exists at least a function $u\in H_{\infty,per}%
^{1}(\mathbb{R}^{d})$ such that
\begin{equation}
\Delta u=g\text{ in }\mathbb{R}^{d},\ M(u)=0. \label{2.9}%
\end{equation}

\end{lemma}

\begin{proof}
It is enough to solve (\ref{2.9}) in the sense of the duality arising from the
mean value. So, we consider the variational form of (\ref{2.9}):
\begin{equation}
M(\nabla u\cdot\nabla v)=-M(gv)\text{ for all }v\in H_{\infty,per}%
^{1}(\mathbb{R}^{d})\text{ with }M(v)=0\text{.} \label{2.10}%
\end{equation}
The Equation (\ref{2.10}) will possess a solution once we will show that the
linear functional $l:v\mapsto-M(gv)$ is continuous on $\mathcal{H}=\{v\in
H_{\infty,per}^{1}(\mathbb{R}^{d}):M(v)=0\}$ endowed with the gradient
seminorm. This will lead to the continuity of the associated form defined on
$\mathcal{H}_{\#}^{1}(\mathbb{R}^{d})$ by $L(\overset{\circ}{v})=-M(gv)$ and
will yield the existence of a unique

$\overset{\circ}{u}\in\mathcal{H}_{\#}^{1}(\mathbb{R}^{d})$ with $u\in
H_{\infty,per}^{1}(\mathbb{R}^{d})$ solving (\ref{2.9}). The continuity of $l$
will stem from a Poincar\'{e}-type inequality satisfied by the elements of
$\mathcal{H}$. Indeed, let $v\in\mathcal{H}$, then $\ell(\overset{\circ}%
{v})\in H_{per}^{1}(Y)/\mathbb{R}$, so that by the Poincar\'{e}-Wirtinger
inequality, there exists a positive constant $C$ independent of $v$ such that
\begin{equation}
\left\Vert \ell(\overset{\circ}{v})\right\Vert _{L^{2}(Y)}\leq C\left\Vert
\nabla\ell(\overset{\circ}{v})\right\Vert _{L^{2}(Y)}. \label{2.10'}%
\end{equation}
But by (\ref{8.5}) we have $\nabla\ell(\overset{\circ}{v})=\ell(\overline
{\nabla}\overset{\circ}{v})=\ell(\varrho(\nabla v))$, and using the fact that
$\ell$ is an isometry from $\mathcal{L}_{\infty,per}^{2}(\mathbb{R}^{d})$ into
$L_{per}^{2}(Y)$ together with the equality $\left\Vert \overset{\circ}%
{v}\right\Vert _{\mathcal{L}_{\infty,per}^{2}(\mathbb{R}^{d})}=\left\Vert
v\right\Vert _{L_{\infty,per}^{2}(\mathbb{R}^{d})}$, we get readily
\[
\left\Vert v\right\Vert _{L_{\infty,per}^{2}(\mathbb{R}^{d})}\leq C\left\Vert
\nabla v\right\Vert _{L_{\infty,per}^{2}(\mathbb{R}^{d})}\text{ for all }%
v\in\mathcal{H}%
\]
where the constant $C$ is the same as in (\ref{2.10'}).
\end{proof}

\begin{lemma}
\label{l5.2}Let $\chi_{j}$ ($0\leq j\leq d$) be the solutions of the corrector
problems \emph{(\ref{2.0})} and \emph{(\ref{2.0'})}, and let $\Omega$ be a
$\mathcal{C}^{1,1}$ open bounded domain in $\mathbb{R}^{d}$. Then it holds
that
\begin{equation}
\int_{\Omega}\left\vert \left(  \nabla_{y}\chi_{j}\right)  \left(  \frac
{x}{\varepsilon}\right)  w(x)\right\vert ^{2}dx\leq C\int_{\Omega}(\left\vert
w\right\vert ^{2}+\varepsilon^{2}\left\vert \nabla w\right\vert ^{2}%
)dx,\text{\ for all }w\in H^{1}(\Omega) \label{5.30}%
\end{equation}
where $C=C(A,\Omega,d)>0$.
\end{lemma}

\begin{proof}
In view of \cite[ Chapter III, Theorem 13.1]{LU1968}, there is a positive
constant $C=C(A,d)$ such that $\left\Vert \chi_{j}\right\Vert _{L^{\infty
}(\mathbb{R}^{d})}\leq C$ for all $0\leq j\leq d$. With this in mind, the
proof of the lemma follows directly from the proof of \cite[Lemma
3.4]{Zhikov1}.
\end{proof}

We end this subsection with a result whose proof can be found in
\cite{Suslina} and which will be useful in the next section.

\begin{lemma}
[{\cite[Lemma 5.1]{Suslina}}]\label{l5.4}Let $\Omega$ be as in Lemma
\emph{\ref{l5.2}}. Then there exists $\varepsilon_{0}\in(0,1]$ depending on
$\Omega$ such that, for any $u\in H^{1}(\Omega)$,
\begin{equation}
\int_{\Gamma_{\varepsilon}}\left\vert u\right\vert ^{2}dx\leq C\varepsilon
\left\Vert u\right\Vert _{L^{2}(\Omega)}\left\Vert u\right\Vert _{H^{1}%
(\Omega)}\text{, }0<\varepsilon\leq\varepsilon_{0} \label{5.34}%
\end{equation}
where $C=C(\Omega)$ and $\Gamma_{\varepsilon}=\Omega_{\varepsilon}\cap\Omega$
with $\Omega_{\varepsilon}=\{x\in\mathbb{R}^{d}:\mathrm{dist}(x,\partial
\Omega)<\varepsilon\}$.
\end{lemma}

We recall that $\Omega_{\varepsilon}=\{x\in\mathbb{R}^{d}:\mathrm{dist}%
(x,\partial\Omega)<\varepsilon\}$ is a tubular neighborhood of $\partial
\Omega$.

\subsection{Convergence rates in $H^{1}$}

In this section, we assume that $\Omega$ is a bounded domain of class
$\mathcal{C}^{1,1}$ in $\mathbb{R}^{d}$. This allows to define a linear
continuous extension operator (see \cite[Chapter 5, Theorem 5.22]{Adams})
\begin{equation}
\widetilde{}:H^{2}(\Omega)\rightarrow H^{2}(\mathbb{R}^{d}) \label{6.16}%
\end{equation}
satisfying
\begin{equation}
\widetilde{u}=u\text{ in }\Omega\text{ and }\left\Vert \widetilde
{u}\right\Vert _{H^{2}(\mathbb{R}^{d})}\leq C\left\Vert u\right\Vert
_{H^{2}(\Omega)}, \label{7.0}%
\end{equation}
where $C>0$ depends only on $\Omega$. Following \cite{Shen2} we define the
smoothing operator $S_{\varepsilon}$ as follows. Fix $\theta\in\mathcal{C}%
_{0}^{\infty}(B_{\frac{1}{4}})$ such that $\theta\geq0$ and $\int
_{\mathbb{R}^{d}}\theta dy=1$. For fixed $\varepsilon>0$, define
\begin{equation}
S_{\varepsilon}(f)(x)=f\ast\theta_{\varepsilon}(y)=\int_{\mathbb{R}^{d}%
}f(y-x)\theta_{\varepsilon}(x)dx \label{7.0'}%
\end{equation}
where $\theta_{\varepsilon}(x)=\varepsilon^{-d}\theta(x/\varepsilon)$. Then
the following properties of $S_{\varepsilon}$ are easy consequences of the
convolution operator and Fourier transform (see e.g., \cite[Lemmas 2.1 and
2.2]{Shen2}).

\begin{lemma}
\label{l6.1}Let $f\in L^{p}(\mathbb{R}^{d})$ for some $1\leq p<\infty$. Then
for any $g\in L_{loc}^{p}(\mathbb{R}^{d})$,
\begin{equation}
\left\Vert g^{\varepsilon}S_{\varepsilon}(f)\right\Vert _{L^{p}(\mathbb{R}%
^{d})}\leq C\sup_{x\in\mathbb{R}^{d}}\left(
\mathchoice {{\setbox0=\hbox{$\displaystyle{\textstyle
-}{\int}$ } \vcenter{\hbox{$\textstyle -$
}}\kern-.6\wd0}}{{\setbox0=\hbox{$\textstyle{\scriptstyle -}{\int}$ } \vcenter{\hbox{$\scriptstyle -$
}}\kern-.6\wd0}}{{\setbox0=\hbox{$\scriptstyle{\scriptscriptstyle -}{\int}$
} \vcenter{\hbox{$\scriptscriptstyle -$
}}\kern-.6\wd0}}{{\setbox0=\hbox{$\scriptscriptstyle{\scriptscriptstyle
-}{\int}$ } \vcenter{\hbox{$\scriptscriptstyle -$ }}\kern-.6\wd0}}\!\int
_{B_{1}(x)}\left\vert g\right\vert ^{p}\right)  ^{\frac{1}{p}}\left\Vert
f\right\Vert _{L^{p}(\mathbb{R}^{d})} \label{7.1}%
\end{equation}
where $g^{\varepsilon}(x)=g(x/\varepsilon)$. If further $f\in W^{1,p}%
(\mathbb{R}^{d})$ for some $1<p<\infty$, then
\begin{equation}
\left\Vert S_{\varepsilon}(f)-f\right\Vert _{L^{p}(\mathbb{R}^{d})}\leq
C\varepsilon\left\Vert \nabla f\right\Vert _{L^{p}(\mathbb{R}^{d})}.
\label{7.2}%
\end{equation}
The constants $C$ above depend only on $d$.
\end{lemma}

The following result will be used in the sequel.

\begin{lemma}
\label{l2.5}$Let$ $b\in L_{\infty,per}^{2}(\mathbb{R}^{d})$ be such that
$M(b)=0$. Then for any $u\in H^{1}(\mathbb{R}^{d})$,
\begin{equation}
\left\Vert b^{\varepsilon}S_{\varepsilon}(u)\right\Vert _{L^{2}(\mathbb{R}%
^{d})}\leq C\varepsilon\left\Vert u\right\Vert _{H^{1}(\mathbb{R}^{d})}
\label{2.11}%
\end{equation}
where $b^{\varepsilon}(x)=b(x/\varepsilon)$ and $C$ is independent of $u$ and
$\varepsilon$.
\end{lemma}

\begin{proof}
Appealing to Lemma \ref{l2.4}, let $w\in H_{\infty,per}^{1}(Y)$ with
$h=\nabla_{y}w\in L_{\infty,per}^{2}(\mathbb{R}^{d})$ be such that $\Delta
_{y}w=b$ in $\mathbb{R}^{d}$. Then $\operatorname{div}_{y}h=b$, so that
$b^{\varepsilon}=\varepsilon\operatorname{div}_{x}h^{\varepsilon}$ (where
$h^{\varepsilon}(x)=h(x/\varepsilon)$). For any $\varphi\in\mathcal{C}%
_{0}^{\infty}(\mathbb{R}^{d})$ we have
\begin{align*}
\left\vert \int_{\mathbb{R}^{d}}b^{\varepsilon}S_{\varepsilon}(u)\varphi
dx\right\vert  &  =\left\vert \varepsilon\int_{\mathbb{R}^{d}}%
(\operatorname{div}_{x}h^{\varepsilon})S_{\varepsilon}(u)\varphi dx\right\vert
=\left\vert -\varepsilon\int_{\mathbb{R}^{d}}h^{\varepsilon}\nabla
(S_{\varepsilon}(u)\varphi)dx\right\vert \\
&  =\varepsilon\left\vert \int_{\mathbb{R}^{d}}h^{\varepsilon}(\varphi\nabla
S_{\varepsilon}(u)+S_{\varepsilon}(u)\nabla\varphi)dx\right\vert \\
&  \leq\varepsilon\,C\sup_{x\in\mathbb{R}^{d}}\left(
\mathchoice {{\setbox0=\hbox{$\displaystyle{\textstyle
-}{\int}$ } \vcenter{\hbox{$\textstyle -$
}}\kern-.6\wd0}}{{\setbox0=\hbox{$\textstyle{\scriptstyle -}{\int}$ } \vcenter{\hbox{$\scriptstyle -$
}}\kern-.6\wd0}}{{\setbox0=\hbox{$\scriptstyle{\scriptscriptstyle -}{\int}$
} \vcenter{\hbox{$\scriptscriptstyle -$
}}\kern-.6\wd0}}{{\setbox0=\hbox{$\scriptscriptstyle{\scriptscriptstyle
-}{\int}$ } \vcenter{\hbox{$\scriptscriptstyle -$ }}\kern-.6\wd0}}\!\int
_{B_{1}(x)}\left\vert h\right\vert ^{2}\right)  ^{\frac{1}{2}}\left\Vert
u\right\Vert _{H^{1}(\mathbb{R}^{d})}\left\Vert \varphi\right\Vert
_{H^{1}(\mathbb{R}^{d})}\\
&  \leq\varepsilon\,C\left\Vert u\right\Vert _{H^{1}(\mathbb{R}^{d}%
)}\left\Vert \varphi\right\Vert _{H^{1}(\mathbb{R}^{d})}.
\end{align*}
where above for the first inequality, we have used (\ref{7.1}) and the fact
that $\nabla\circ S_{\varepsilon}=S_{\varepsilon}\circ\nabla$. It follows that
$\left\Vert b^{\varepsilon}S_{\varepsilon}(u)\right\Vert _{H^{-1}%
(\mathbb{R}^{d})}\leq C\varepsilon\left\Vert u\right\Vert _{H^{1}%
(\mathbb{R}^{d})}$. However, as $b^{\varepsilon}S_{\varepsilon}(u)\in
L^{2}(\mathbb{R}^{d})$ (see (\ref{7.1})), it readily follows that $\left\Vert
b^{\varepsilon}S_{\varepsilon}(u)\right\Vert _{H^{-1}(\mathbb{R}^{d}%
)}=\left\Vert b^{\varepsilon}S_{\varepsilon}(u)\right\Vert _{L^{2}%
(\mathbb{R}^{d})}$ (see \cite[page 74]{Adams}) and so \eqref{2.11} holds.
\end{proof}

Let $u_{\varepsilon}$, $u_{0}\in H_{0}^{1}(\Omega)$ be weak solutions of
(\ref{*10}) and (\ref{1.4'}) respectively. We assume further that $u_{0}\in
H^{2}(\Omega)$. We denote the first order approximation of $u_{\varepsilon}$
by
\begin{equation}
v_{\varepsilon}=u_{0}+\varepsilon\chi^{\varepsilon}S_{\varepsilon}%
(\nabla\widetilde{u}_{0})+\varepsilon\chi_{0}^{\varepsilon}S_{\varepsilon
}(\widetilde{u}_{0}) \label{5.1}%
\end{equation}
where $\chi^{\varepsilon}(x)=\chi(x/\varepsilon)=(\chi_{j}(x/\varepsilon
))_{1\leq j\leq d}$ and $\chi_{0}^{\varepsilon}(x)=\chi_{0}(x/\varepsilon)$
(for $x\in\Omega$) with $\chi_{j}$ and $\chi_{0}$ being solutions of
(\ref{2.0'}) and (\ref{2.0}) respectively. Let $w_{\varepsilon}=u_{\varepsilon
}-u_{0}-\varepsilon\chi^{\varepsilon}S_{\varepsilon}(\nabla\widetilde{u}%
_{0})-\varepsilon\chi_{0}^{\varepsilon}S_{\varepsilon}(\widetilde{u}%
_{0})+z_{\varepsilon}$, where $z_{\varepsilon}\in H^{1}(\Omega)$ is defined
by
\begin{equation}
\mathcal{P}_{\varepsilon}z_{\varepsilon}=0\text{ in }\Omega,\ z_{\varepsilon
}=\varepsilon\chi^{\varepsilon}S_{\varepsilon}(\nabla\widetilde{u}%
_{0})+\varepsilon\chi_{0}^{\varepsilon}S_{\varepsilon}(\widetilde{u}%
_{0})\text{ on }\partial\Omega. \label{5.3}%
\end{equation}
Then setting
\begin{align*}
\mathcal{A}_{0}(y)  &  =\widehat{A}-A(y)(I_{d}+\nabla\chi(y))=(A_{j})_{1\leq
j\leq d},\\
V_{1}(y)  &  =\widehat{V}-(A(y)\nabla\chi_{0}(y)+V(y))=(V_{1,j})_{1\leq j\leq
d}\\
B_{1}(y)  &  =\widehat{B}-B(y)(\nabla\chi(y)+I_{d})\\
a_{0,1}(y)  &  =\widehat{a}_{0}-(B(y)\nabla\chi_{0}(y)+a_{0}(y))
\end{align*}
with $A_{j}=(A_{ij})_{1\leq i\leq d}$,\ \ $A_{ij}(y)=\widehat{a}_{ij}%
-a_{ij}(y)-\sum_{k=1}^{d}a_{ik}(y)\frac{\partial\chi_{j}}{\partial y_{k}}%
(y)$,\ $\widehat{A}=(\widehat{a}_{ij})_{1\leq i,j\leq d}$, we have the following

\begin{lemma}
\label{l6.2}Let $u_{\varepsilon}$, $u_{0}$ be the weak solutions of
\emph{(\ref{*10})} and \emph{(\ref{1.4'})} respectively. Assume $u_{0}\in
H^{2}(\Omega)$. Then
\begin{equation}%
\begin{array}
[c]{l}%
\mathcal{P}_{\varepsilon}w_{\varepsilon}=-\operatorname{div}(\mathcal{A}%
_{0}^{\varepsilon}S_{\varepsilon}(\nabla\widetilde{u}_{0})+V_{1}^{\varepsilon
}S_{\varepsilon}(\widetilde{u}_{0}))\\
-\operatorname{div}[(\widehat{A}-A^{\varepsilon})(\nabla u_{0}-S_{\varepsilon
}(\nabla\widetilde{u}_{0}))+(\widehat{V}-V^{\varepsilon})(u_{0}-S_{\varepsilon
}(\widetilde{u}_{0}))]+B_{1}^{\varepsilon}S_{\varepsilon}(\nabla\widetilde
{u}_{0})\\
+a_{0,1}^{\varepsilon}S_{\varepsilon}(\widetilde{u}_{0})+(\widehat
{B}-B^{\varepsilon})(\nabla u_{0}-S_{\varepsilon}(\nabla\widetilde{u}%
_{0}))+(\widehat{a}_{0}-a_{0}^{\varepsilon})(u_{0}-S_{\varepsilon}%
(\widetilde{u}_{0}))\\
+\varepsilon\operatorname{div}[\chi^{\varepsilon}A^{\varepsilon}\nabla
S_{\varepsilon}(\nabla\widetilde{u}_{0})+\chi_{0}^{\varepsilon}\nabla
S_{\varepsilon}(\widetilde{u}_{0})+\chi^{\varepsilon}V^{\varepsilon
}S_{\varepsilon}(\nabla\widetilde{u}_{0})+\chi_{0}^{\varepsilon}%
V^{\varepsilon}S_{\varepsilon}(\widetilde{u}_{0})]\\
-\varepsilon\lbrack\chi^{\varepsilon}B^{\varepsilon}\nabla S_{\varepsilon
}(\nabla\widetilde{u}_{0})+\chi_{0}^{\varepsilon}B^{\varepsilon}\nabla
S_{\varepsilon}(\widetilde{u}_{0})+a_{0}^{\varepsilon}\chi^{\varepsilon
}S_{\varepsilon}(\nabla\widetilde{u}_{0})\\
+a_{0}^{\varepsilon}\chi_{0}^{\varepsilon}S_{\varepsilon}(\widetilde{u}%
_{0})+\mu\chi^{\varepsilon}S_{\varepsilon}(\widetilde{u}_{0})+\mu\chi
_{0}^{\varepsilon}S_{\varepsilon}(\widetilde{u}_{0})].
\end{array}
\label{7.3}%
\end{equation}

\end{lemma}

\begin{proof}
By direct calculations we get
\begin{align*}
&  -\operatorname{div}(A^{\varepsilon}w_{\varepsilon}+V^{\varepsilon
}w_{\varepsilon})+B^{\varepsilon}\nabla w_{\varepsilon}+a_{0}^{\varepsilon
}w_{\varepsilon}+\mu w_{\varepsilon}\\
&  =-\operatorname{div}(A^{\varepsilon}\nabla u_{\varepsilon}+V^{\varepsilon
}u_{\varepsilon})+B^{\varepsilon}\nabla u_{\varepsilon}+a_{0}^{\varepsilon
}u_{\varepsilon}+\mu u_{\varepsilon}\\
&  +\operatorname{div}(A^{\varepsilon}\nabla v_{\varepsilon}+V^{\varepsilon
}v_{\varepsilon})-(B^{\varepsilon}\nabla v_{\varepsilon}+a_{0}^{\varepsilon
}v_{\varepsilon}+\mu v_{\varepsilon})\\
&  =-\operatorname{div}(\widehat{A}\nabla u_{0}+\widehat{V}u_{0})+\widehat
{B}\nabla u_{0}+\widehat{a}_{0}u_{0}+\mu u_{0}+\operatorname{div}%
(A^{\varepsilon}\nabla v_{\varepsilon}+V^{\varepsilon}v_{\varepsilon})\\
&  -(B^{\varepsilon}\nabla v_{\varepsilon}+a_{0}^{\varepsilon}v_{\varepsilon
}+\mu v_{\varepsilon})\\
&  =-\operatorname{div}\left[  (\widehat{A}\nabla u_{0}-A^{\varepsilon}\nabla
v_{\varepsilon})+(\widehat{V}u_{0}-V^{\varepsilon}v_{\varepsilon})\right]
+\widehat{B}\nabla u_{0}-B^{\varepsilon}\nabla v_{\varepsilon}\\
&  +\widehat{a}_{0}u_{0}-a_{0}^{\varepsilon}v_{\varepsilon}+\mu(u_{0}%
-v_{\varepsilon}),
\end{align*}
where above, we have used in the second equality, the fact that $\mathcal{P}%
_{\varepsilon}u_{\varepsilon}=\mathcal{P}_{0}u_{0}$. But
\begin{align*}
&  -\operatorname{div}\left[  (\widehat{A}\nabla u_{0}-A^{\varepsilon}\nabla
v_{\varepsilon})+(\widehat{V}u_{0}-V^{\varepsilon}v_{\varepsilon})\right] \\
&  =-\operatorname{div}(\mathcal{A}_{0}^{\varepsilon}S_{\varepsilon}%
(\nabla\widetilde{u}_{0})+V_{1}^{\varepsilon}S_{\varepsilon}(\widetilde{u}%
_{0}))\\
&  -\operatorname{div}\left(  (\widehat{A}-A^{\varepsilon})(\nabla
u_{0}-S_{\varepsilon}(\nabla\widetilde{u}_{0}))+(\widehat{V}-V^{\varepsilon
})(u_{0}-S_{\varepsilon}(\widetilde{u}_{0}))\right) \\
&  +\varepsilon\operatorname{div}\left(  \chi^{\varepsilon}A^{\varepsilon
}\nabla S_{\varepsilon}(\nabla\widetilde{u}_{0}\right)  +\chi_{0}%
^{\varepsilon}A^{\varepsilon}\nabla S_{\varepsilon}(\widetilde{u}_{0}%
)+\chi^{\varepsilon}V^{\varepsilon}S_{\varepsilon}(\nabla\widetilde{u}%
_{0})+\chi_{0}^{\varepsilon}V^{\varepsilon}S_{\varepsilon}(\widetilde{u}_{0}))
\end{align*}
and
\begin{align*}
&  \widehat{B}\nabla u_{0}-B^{\varepsilon}\nabla v_{\varepsilon}+\widehat
{a}_{0}u_{0}-a_{0}^{\varepsilon}v_{\varepsilon}+\mu(u_{0}-v_{\varepsilon})\\
&  =(\widehat{B}-B^{\varepsilon})(\nabla u_{0}-S_{\varepsilon}(\nabla
\widetilde{u}_{0}))+B_{1}^{\varepsilon}S_{\varepsilon}(\nabla\widetilde{u}%
_{0})+(\widehat{a}_{0}-a_{0}^{\varepsilon})(u_{0}-S_{\varepsilon}%
(\widetilde{u}_{0}))+a_{0,1}^{\varepsilon}S_{\varepsilon}(\widetilde{u}_{0})\\
&  -\varepsilon\lbrack\chi^{\varepsilon}B^{\varepsilon}\nabla S_{\varepsilon
}(\nabla\widetilde{u}_{0})+\chi_{0}^{\varepsilon}B^{\varepsilon}\nabla
S_{\varepsilon}(\widetilde{u}_{0})+a_{0}^{\varepsilon}\chi^{\varepsilon
}S_{\varepsilon}(\nabla\widetilde{u}_{0})+a_{0}^{\varepsilon}\chi
_{0}^{\varepsilon}S_{\varepsilon}(\widetilde{u}_{0})\\
&  +\mu(\chi^{\varepsilon}S_{\varepsilon}(\nabla\widetilde{u}_{0})+\chi
_{0}^{\varepsilon}S_{\varepsilon}(\widetilde{u}_{0})].
\end{align*}
The result follows thereby.
\end{proof}

With the above result in mind, our next aim is to prove the $H^{1}$-rate of
convergence stated in the following result.

\begin{proposition}
\label{p6.2}Let $\Omega$ be a $\mathcal{C}^{1,1}$ bounded domain in
$\mathbb{R}^{d}$. Suppose \emph{(\ref{1.2})}, \emph{(\ref{2.1})},
\emph{(\ref{2.2'})} and \emph{(\ref{2.2})} hold. Let $u_{\varepsilon}$ and
$u_{0}$ be the weak solutions of \emph{(\ref{*10})} and \emph{(\ref{1.4'})}
respectively. Under the further assumption that $u_{0}\in H^{2}(\Omega)$,
there exists $C=C(d,\Omega,\alpha,\beta,\alpha_{0})>0$ such that
\begin{equation}
\left\Vert u_{\varepsilon}-u_{0}-\varepsilon\chi^{\varepsilon}S_{\varepsilon
}(\nabla\widetilde{u}_{0})-\varepsilon\chi_{0}^{\varepsilon}S_{\varepsilon
}(\widetilde{u}_{0})+z_{\varepsilon}\right\Vert _{H^{1}(\Omega)}\leq
C\varepsilon\left\Vert u_{0}\right\Vert _{H^{2}(\Omega)}. \label{20}%
\end{equation}

\end{proposition}

\begin{proof}
First of all, let $\mathbf{g}\in(L_{\infty,per}^{2}(\mathbb{R}^{d}))^{d}$ be
solenoidal ($\operatorname{div}\mathbf{g}=0$) with $M(\mathbf{g})=0$. Then in
view of Lemma \ref{l2.4}, we can show that there exists a skew-symmetric
matrix $G$ with entries in $H_{\infty,per}^{1}(\mathbb{R}^{d})$ such that
$\mathbf{g}=\operatorname{div}G$. Moreover it holds that
\begin{equation}
\sup_{x\in\mathbb{R}^{d}}\left(
\mathchoice {{\setbox0=\hbox{$\displaystyle{\textstyle
-}{\int}$ } \vcenter{\hbox{$\textstyle -$
}}\kern-.6\wd0}}{{\setbox0=\hbox{$\textstyle{\scriptstyle -}{\int}$ } \vcenter{\hbox{$\scriptstyle -$
}}\kern-.6\wd0}}{{\setbox0=\hbox{$\scriptstyle{\scriptscriptstyle -}{\int}$
} \vcenter{\hbox{$\scriptscriptstyle -$
}}\kern-.6\wd0}}{{\setbox0=\hbox{$\scriptscriptstyle{\scriptscriptstyle
-}{\int}$ } \vcenter{\hbox{$\scriptscriptstyle -$ }}\kern-.6\wd0}}\!\int
_{B_{1}(x)}\left\vert G\right\vert ^{2}\right)  ^{\frac{1}{2}}\leq
C=C(\mathbf{g},d). \label{21}%
\end{equation}

Bearing this in mind, let us first analyze the terms $\mathcal{A}_{0}%
,V_{1},B_{1}$ and $a_{0,1}$. Starting with $\mathcal{A}_{0}=(A_{j})_{1\leq
j\leq d}$ where $A_{j}=(A_{ij})_{1\leq i\leq d}$ with $A_{ij}(y)=\widehat
{a}_{ij}-a_{ij}(y)-\sum_{k=1}^{d}a_{ik}(y)\frac{\partial\chi_{j}}{\partial
y_{k}}(y)$, we observe that $A_{ij}\in L_{\infty,per}^{2}(\mathbb{R}^{d})$.
Thus $A_{j}\in(L_{\infty,per}^{2}(\mathbb{R}^{d}))^{d}$ with $M(A_{j})=0$ and
$\operatorname{div}A_{j}=0$. Appealing to what has being said in the first
lines of this proof, we derive the existence of a skew-symmetric matrix
$G_{j}$ with entries in $H_{\infty,per}^{1}(\mathbb{R}^{d})$ such that%
\[
A_{j}=\operatorname{div}G_{j}\text{ and }\sup_{x\in\mathbb{R}^{d}}\left(
\mathchoice {{\setbox0=\hbox{$\displaystyle{\textstyle
-}{\int}$ } \vcenter{\hbox{$\textstyle -$
}}\kern-.6\wd0}}{{\setbox0=\hbox{$\textstyle{\scriptstyle -}{\int}$ } \vcenter{\hbox{$\scriptstyle -$
}}\kern-.6\wd0}}{{\setbox0=\hbox{$\scriptstyle{\scriptscriptstyle -}{\int}$
} \vcenter{\hbox{$\scriptscriptstyle -$
}}\kern-.6\wd0}}{{\setbox0=\hbox{$\scriptscriptstyle{\scriptscriptstyle
-}{\int}$ } \vcenter{\hbox{$\scriptscriptstyle -$ }}\kern-.6\wd0}}\!\int
_{B_{1}(x)}\left\vert G_{j}\right\vert ^{2}\right)  ^{\frac{1}{2}}\leq C.
\]
As for $V_{1}$, we repeat the same process to obtain a skew-symmetric matrix
$H\in(H_{\infty,per}^{1}(Y))^{d\times d}$ such that
\[
V_{1}=\operatorname{div}H\text{ and }\sup_{x\in\mathbb{R}^{d}}\left(
\mathchoice {{\setbox0=\hbox{$\displaystyle{\textstyle
-}{\int}$ } \vcenter{\hbox{$\textstyle -$
}}\kern-.6\wd0}}{{\setbox0=\hbox{$\textstyle{\scriptstyle -}{\int}$ } \vcenter{\hbox{$\scriptstyle -$
}}\kern-.6\wd0}}{{\setbox0=\hbox{$\scriptstyle{\scriptscriptstyle -}{\int}$
} \vcenter{\hbox{$\scriptscriptstyle -$
}}\kern-.6\wd0}}{{\setbox0=\hbox{$\scriptscriptstyle{\scriptscriptstyle
-}{\int}$ } \vcenter{\hbox{$\scriptscriptstyle -$ }}\kern-.6\wd0}}\!\int
_{B_{1}(x)}\left\vert H\right\vert ^{2}\right)  ^{\frac{1}{2}}\leq C.
\]
Set $B_{1}=(b_{i}^{1})_{1\leq i\leq d}$ where $b_{i}^{1}\in L_{\infty,per}%
^{2}(\mathbb{R}^{d})$ with $M(b_{i}^{1})=0$. Then we infer from Lemma
\ref{l2.4} that there is a vector $\boldsymbol{h}_{i}\in(H_{\infty,per}%
^{1}(\mathbb{R}^{d}))^{d}$ for $0\leq i\leq d$, such that
\[
b_{i}^{1}=\operatorname{div}\boldsymbol{h}_{i}\ \ (1\leq i\leq d)\text{ and
}a_{0,1}=\operatorname{div}\boldsymbol{h}_{0}%
\]
and
\[
\sup_{x\in\mathbb{R}^{d}}\left(
\mathchoice {{\setbox0=\hbox{$\displaystyle{\textstyle
-}{\int}$ } \vcenter{\hbox{$\textstyle -$
}}\kern-.6\wd0}}{{\setbox0=\hbox{$\textstyle{\scriptstyle -}{\int}$ } \vcenter{\hbox{$\scriptstyle -$
}}\kern-.6\wd0}}{{\setbox0=\hbox{$\scriptstyle{\scriptscriptstyle -}{\int}$
} \vcenter{\hbox{$\scriptscriptstyle -$
}}\kern-.6\wd0}}{{\setbox0=\hbox{$\scriptscriptstyle{\scriptscriptstyle
-}{\int}$ } \vcenter{\hbox{$\scriptscriptstyle -$ }}\kern-.6\wd0}}\!\int
_{B_{1}(x)}\left\vert \boldsymbol{h}_{i}\right\vert ^{2}\right)  ^{\frac{1}%
{2}}\leq C,\ \ 0\leq i\leq d.
\]
Recalling that $G_{j}$ and $H$ are skew-symmetric, we get that
\[
-\operatorname{div}(\mathcal{A}_{0}^{\varepsilon}S_{\varepsilon}%
(\nabla\widetilde{u}_{0})+V_{1}^{\varepsilon}S_{\varepsilon}(\widetilde{u}%
_{0}))=-\varepsilon\operatorname{div}r_{1}^{\varepsilon}%
\]
where
\[
r_{1}^{\varepsilon}(x)=\sum_{j=1}^{d}G_{j}\left(  \frac{x}{\varepsilon
}\right)  \nabla S_{\varepsilon}\left(  \frac{\partial\widetilde{u}_{0}%
}{\partial x_{j}}\right)  +H\left(  \frac{x}{\varepsilon}\right)  \nabla
S_{\varepsilon}(\widetilde{u}_{0}).
\]
Next we appeal to Lemma \ref{l2.5} and we use (\ref{21}) in conjunction with
(\ref{7.1}) [in Lemma \ref{l6.1}] to obtain
\begin{equation}
\left\Vert B_{1}^{\varepsilon}S_{\varepsilon}(\nabla\widetilde{u}%
_{0})\right\Vert _{L^{2}(\Omega)}\leq C\varepsilon\left\Vert u_{0}\right\Vert
_{H^{2}(\Omega)}\text{ and }\left\Vert a_{0,1}^{\varepsilon}S_{\varepsilon
}(\widetilde{u}_{0})\right\Vert _{L^{2}(\Omega)}\leq C\varepsilon\left\Vert
u_{0}\right\Vert _{H^{1}(\Omega)}. \label{22}%
\end{equation}
Dealing with the term $(\widehat{A}-A^{\varepsilon})(\nabla u_{0}%
-S_{\varepsilon}(\nabla\widetilde{u}_{0}))$ and the related ones corresponding
to other coefficients of equations (\ref{*10}) and (\ref{1.4'}), we make use
of (\ref{7.2}) in Lemma \ref{l6.1} to get
\begin{equation}%
\begin{array}
[c]{l}%
\left\Vert (\widehat{A}-A^{\varepsilon})(\nabla u_{0}-S_{\varepsilon}%
(\nabla\widetilde{u}_{0}))\right\Vert _{L^{2}(\Omega)}\leq C\varepsilon
\left\Vert u_{0}\right\Vert _{H^{2}(\Omega)}\\
\left\Vert (\widehat{V}-V^{\varepsilon})(u_{0}-S_{\varepsilon}(\widetilde
{u}_{0}))\right\Vert _{L^{2}(\Omega)}\leq C\varepsilon\left\Vert
u_{0}\right\Vert _{H^{1}(\Omega)}\\
\left\Vert (\widehat{B}-B^{\varepsilon})(\nabla u_{0}-S_{\varepsilon}%
(\nabla\widetilde{u}_{0}))\right\Vert _{L^{2}(\Omega)}\leq C\varepsilon
\left\Vert u_{0}\right\Vert _{H^{2}(\Omega)}\\
\left\Vert (\widehat{a}_{0}-a_{0}^{\varepsilon})(u_{0}-S_{\varepsilon
}(\widetilde{u}_{0}))\right\Vert _{L^{2}(\Omega)}\leq C\varepsilon\left\Vert
u_{0}\right\Vert _{H^{1}(\Omega)}.
\end{array}
\label{24}%
\end{equation}
Now, coming back to the equality in Lemma \ref{l6.2}, and using the inequality
(\ref{*11}) with $u_{\varepsilon}$ being replaced by $w_{\varepsilon}$, and
$f$ and $F$ replaced respectively by
\[%
\begin{array}
[c]{l}%
f=B_{1}^{\varepsilon}S_{\varepsilon}(\nabla\widetilde{u}_{0})+a_{0,1}%
^{\varepsilon}S_{\varepsilon}(\widetilde{u}_{0})+(\widehat{B}-B^{\varepsilon
})(\nabla u_{0}-S_{\varepsilon}(\nabla\widetilde{u}_{0}))\\
+(\widehat{a}_{0}-a_{0}^{\varepsilon})(u_{0}-S_{\varepsilon}(\widetilde{u}%
_{0}))-\varepsilon\lbrack\chi^{\varepsilon}B^{\varepsilon}\nabla
S_{\varepsilon}(\nabla\widetilde{u}_{0})+\chi_{0}^{\varepsilon}B^{\varepsilon
}\nabla S_{\varepsilon}(\widetilde{u}_{0})\\
+a_{0}^{\varepsilon}\chi^{\varepsilon}S_{\varepsilon}(\nabla\widetilde{u}%
_{0})+a_{0}^{\varepsilon}\chi_{0}^{\varepsilon}S_{\varepsilon}(\widetilde
{u}_{0})+\mu\chi^{\varepsilon}S_{\varepsilon}(\widetilde{u}_{0})+\mu\chi
_{0}^{\varepsilon}S_{\varepsilon}(\widetilde{u}_{0})].
\end{array}
\]
and
\[%
\begin{array}
[c]{l}%
F=-(\mathcal{A}_{0}^{\varepsilon}S_{\varepsilon}(\nabla\widetilde{u}%
_{0})+V_{1}^{\varepsilon}S_{\varepsilon}(\widetilde{u}_{0}))\\
-[(\widehat{A}-A^{\varepsilon})(\nabla u_{0}-S_{\varepsilon}(\nabla
\widetilde{u}_{0}))+(\widehat{V}-V^{\varepsilon})(u_{0}-S_{\varepsilon
}(\widetilde{u}_{0}))]\\
+\varepsilon\lbrack\chi^{\varepsilon}A^{\varepsilon}\nabla S_{\varepsilon
}(\nabla\widetilde{u}_{0})+\chi_{0}^{\varepsilon}\nabla S_{\varepsilon
}(\widetilde{u}_{0})+\chi^{\varepsilon}V^{\varepsilon}S_{\varepsilon}%
(\nabla\widetilde{u}_{0})+\chi_{0}^{\varepsilon}V^{\varepsilon}S_{\varepsilon
}(\widetilde{u}_{0})],
\end{array}
\]
and owing to inequalities (\ref{22}), (\ref{24}) and using once again property
(\ref{7.1}), we are led to (\ref{20}). This concludes the proof of the proposition.
\end{proof}

In order to obtain the $H^{1}$ rate of convergence, we need to estimate
$\left\Vert z_{\varepsilon}\right\Vert _{H^{1}(\Omega)}$ in terms of
$\varepsilon$. The next lemma provides us with such an estimate.

\begin{lemma}
\label{l6.5}Let $z_{\varepsilon}$ be defined by \emph{(\ref{5.3})}. There
exist $\varepsilon_{0}>0$ and $C=C(\alpha,\beta,\alpha_{0},\Omega)>0$ such
that
\begin{equation}
\left\Vert z_{\varepsilon}\right\Vert _{H^{1}(\Omega)}\leq C\varepsilon
^{\frac{1}{2}}\left\Vert u_{0}\right\Vert _{H^{2}(\Omega)}\text{ for all
}0<\varepsilon\leq\varepsilon_{0}. \label{27}%
\end{equation}

\end{lemma}

\begin{proof}
Let $\varepsilon_{0}$ be given by Lemma \ref{l5.4}, and fix $\varepsilon>0$
such that $\varepsilon\leq\varepsilon_{0}$. Let $\theta_{\varepsilon}$ be a
cut-off function in a neighborhood of $\partial\Omega$ with support in
$\Omega_{2\varepsilon}$ (a $2\varepsilon$-neighborhood of $\partial\Omega$),
$\Omega_{\varepsilon}$ being defined as in Lemma \ref{l5.4}:
\begin{equation}
\theta_{\varepsilon}\in\mathcal{C}_{0}^{\infty}(\mathbb{R}^{d})\text{,
\textrm{supp}}\theta_{\varepsilon}\subset\Omega_{2\varepsilon}\text{, }%
0\leq\theta_{\varepsilon}\leq1\text{, }\theta_{\varepsilon}=1\text{ on }%
\Omega_{\varepsilon}\text{, }\theta_{\varepsilon}=0\text{ on }\mathbb{R}%
^{d}\backslash\Omega_{2\varepsilon}\text{ and }\varepsilon\left\vert
\nabla\theta_{\varepsilon}\right\vert \leq C. \label{5.35}%
\end{equation}
Define
\begin{equation}
\phi_{\varepsilon}=\varepsilon\theta_{\varepsilon}(\chi^{\varepsilon}\nabla
u_{0}+\chi_{0}^{\varepsilon}u_{0}) \label{5.36}%
\end{equation}
and set $v_{\varepsilon}=z_{\varepsilon}-\phi_{\varepsilon}$. Then obviously
$\phi_{\varepsilon}\in H^{1}(\Omega)$ and $\mathcal{P}_{\varepsilon
}v_{\varepsilon}=-\mathcal{P}_{\varepsilon}\phi_{\varepsilon}$ in $\Omega$,
$v_{\varepsilon}=0$ on $\partial\Omega$, \ so that
\[
\left\Vert v_{\varepsilon}\right\Vert _{H^{1}(\Omega)}\leq C(\left\Vert
A^{\varepsilon}\nabla\phi_{\varepsilon}+V^{\varepsilon}\phi_{\varepsilon
}\right\Vert _{L^{2}(\Omega)}+\left\Vert B^{\varepsilon}\nabla\phi
_{\varepsilon}\right\Vert _{L^{2}(\Omega)}+\left\Vert a_{0}^{\varepsilon}%
\phi_{\varepsilon}\right\Vert _{L^{2}(\Omega)}).
\]
From the above inequality we infer that
\[
\left\Vert z_{\varepsilon}\right\Vert _{H^{1}(\Omega)}\leq\left\Vert
v_{\varepsilon}\right\Vert _{H^{1}(\Omega)}+\left\Vert \phi_{\varepsilon
}\right\Vert _{H^{1}(\Omega)}\leq C(\left\Vert \nabla\phi_{\varepsilon
}\right\Vert _{L^{2}(\Omega)}+\left\Vert \phi_{\varepsilon}\right\Vert
_{L^{2}(\Omega)})
\]
where the constant $C$ depends only on $\alpha$, $\beta$ and $\alpha_{0}$. It
remains to estimate the right-hand side of the last inequality above. To that
end, we have
\begin{align*}
\nabla\phi_{\varepsilon}  &  =\varepsilon\chi^{\varepsilon}\nabla u_{0}%
\nabla\theta_{\varepsilon}+(\nabla_{y}\chi)^{\varepsilon}\nabla u_{0}%
\theta_{\varepsilon}+\varepsilon\chi^{\varepsilon}\theta_{\varepsilon}%
\nabla^{2}u_{0}+\varepsilon\chi_{0}^{\varepsilon}u_{0}\nabla\theta
_{\varepsilon}+(\nabla_{y}\chi_{0})^{\varepsilon}u_{0}\theta_{\varepsilon
}+\varepsilon\chi_{0}^{\varepsilon}\theta_{\varepsilon}\nabla u_{0}\\
&  =\varepsilon(\chi^{\varepsilon}\nabla u_{0}+\chi_{0}^{\varepsilon}%
u_{0})\nabla\theta_{\varepsilon}+((\nabla_{y}\chi)^{\varepsilon}\nabla
u_{0}+(\nabla_{y}\chi_{0})^{\varepsilon}u_{0})\theta_{\varepsilon}%
+\varepsilon(\chi^{\varepsilon}\nabla^{2}u_{0}+\chi_{0}^{\varepsilon}\nabla
u_{0})\theta_{\varepsilon}\\
&  =J_{1}+J_{2}+J_{3}\text{.}%
\end{align*}

As for $J_{1}$, we have,
\begin{align}
\left\Vert J_{1}\right\Vert _{L^{2}(\Omega)}^{2}  &  \leq C\int_{\Gamma
_{2\varepsilon}}(\left\vert \nabla u_{0}\right\vert ^{2}+\left\vert
u_{0}\right\vert ^{2})dx\label{25}\\
&  \leq C\varepsilon\left(  \left\Vert \nabla u_{0}\right\Vert _{L^{2}%
(\Omega)}\left\Vert \nabla u_{0}\right\Vert _{H^{1}(\Omega)}+\left\Vert
u_{0}\right\Vert _{L^{2}(\Omega)}\left\Vert u_{0}\right\Vert _{H^{1}(\Omega
)}\right) \nonumber\\
&  \leq C\varepsilon\left\Vert u_{0}\right\Vert _{H^{2}(\Omega)}^{2}\nonumber
\end{align}
where above in (\ref{25}) we have used the boundedness of $\chi_{j}$
($\left\Vert \chi_{j}\right\Vert _{L^{\infty}(\mathbb{R}^{d})}\leq C$ for all
$0\leq j\leq d$) and the inequality $\left\vert \nabla\theta_{\varepsilon
}\right\vert \leq C\varepsilon^{-1}$ stemming from (\ref{5.35}), to obtain the
first inequality, and (\ref{5.34}) to get the second one.

Dealing with $J_{2}$, we have
\begin{align}
\left\Vert J_{2}\right\Vert _{L^{2}(\Omega)}^{2}  &  \leq C\int_{\Omega
}\left(  \left\vert (\nabla_{y}\chi)^{\varepsilon}\nabla u_{0}\theta
_{\varepsilon}\right\vert ^{2}+\left\vert (\nabla_{y}\chi_{0})^{\varepsilon
}u_{0}\theta_{\varepsilon}\right\vert ^{2}\right)  dx\label{26}\\
&  \leq C\int_{\Omega}\left(  \left\vert \theta_{\varepsilon}\nabla
u_{0}\right\vert ^{2}+\varepsilon^{2}\left\vert \nabla(\theta_{\varepsilon
}\nabla u_{0})\right\vert ^{2}\right)  dx+C\int_{\Omega}\left(  \left\vert
\theta_{\varepsilon}u_{0}\right\vert ^{2}+\varepsilon^{2}\left\vert
\nabla(\theta_{\varepsilon}u_{0})\right\vert ^{2}\right)  dx\nonumber
\end{align}
where we have used (\ref{5.30}) (in Lemma \ref{l5.2}) to obtain the second
inequality in (\ref{26}). Letting $w=u_{0}$ or $\nabla u_{0}$, we have
$\nabla(w\theta_{\varepsilon})=w\nabla\theta_{\varepsilon}+\theta
_{\varepsilon}\nabla w$, and
\begin{align*}
\int_{\Omega}\left\vert \nabla(w\theta_{\varepsilon})\right\vert ^{2}dx  &
\leq C\int_{\Gamma_{2\varepsilon}}\left\vert \nabla\theta_{\varepsilon
}\right\vert ^{2}\left\vert w\right\vert ^{2}dx+C\int_{\Omega}\left\vert
\theta_{\varepsilon}\nabla w\right\vert ^{2}dx\\
&  \leq C\varepsilon^{-1}\left\Vert w\right\Vert _{L^{2}(\Omega)}\left\Vert
w\right\Vert _{H^{1}(\Omega)}+C\int_{\Omega}\left\vert \nabla w\right\vert
^{2}dx.
\end{align*}
Thus
\[
\left\Vert J_{2}\right\Vert _{L^{2}(\Omega)}^{2}\leq C(\varepsilon
+\varepsilon^{2})\left\Vert u_{0}\right\Vert _{H^{2}(\Omega)}^{2}.
\]

Concerning $J_{3}$, we use once again the boundedness of the correctors to
get
\[
\left\Vert J_{3}\right\Vert _{L^{2}(\Omega)}^{2}\leq C\varepsilon
^{2}\left\Vert u_{0}\right\Vert _{H^{2}(\Omega)}^{2}.
\]
It follows that
\begin{equation}
\left\Vert \nabla\phi_{\varepsilon}\right\Vert _{L^{2}(\Omega)}\leq
C\varepsilon^{\frac{1}{2}}\left\Vert u_{0}\right\Vert _{H^{2}(\Omega)}.
\label{eq01}%
\end{equation}
Now, using the obvious inequality
\begin{equation}
\left\Vert \phi_{\varepsilon}\right\Vert _{L^{2}(\Omega)}\leq C\varepsilon
^{\frac{1}{2}}\left\Vert u_{0}\right\Vert _{H^{2}(\Omega)}, \label{eq02}%
\end{equation}
we are led to $\left\Vert z_{\varepsilon}\right\Vert _{H^{1}(\Omega)}\leq
C\varepsilon^{\frac{1}{2}}\left\Vert u_{0}\right\Vert _{H^{2}(\Omega)}$, which
completes the proof.
\end{proof}

Putting together (\ref{20}) and (\ref{27}) associated to the inequality
\[
\left\Vert u_{\varepsilon}-u_{0}-\varepsilon\chi^{\varepsilon}S_{\varepsilon
}(\nabla\widetilde{u}_{0})-\varepsilon\chi_{0}^{\varepsilon}S_{\varepsilon
}(\widetilde{u}_{0})\right\Vert _{H^{1}(\Omega)}\leq C\varepsilon\left\Vert
u_{0}\right\Vert _{H^{2}(\Omega)}+\left\Vert z_{\varepsilon}\right\Vert
_{H^{1}(\Omega)},
\]
we see that we have proven the $H^{1}$-rate of convergence stated in the next result.

\begin{theorem}
\label{t5.1}Let $\Omega$ be a $\mathcal{C}^{1,1}$ bounded domain in
$\mathbb{R}^{d}$. Suppose \emph{(\ref{1.2})}, \emph{(\ref{2.1})},
\emph{(\ref{2.2'})} and \emph{(\ref{2.2})} hold. Let $u_{\varepsilon}$ and
$u_{0}$ be the weak solutions of \emph{(\ref{*10})} and \emph{(\ref{1.4'})}
respectively. Under the further assumption that $u_{0}\in H^{2}(\Omega)$,
there exists $C=C(d,\Omega,\alpha,\beta,\alpha_{0})>0$ such that
\begin{equation}
\left\Vert u_{\varepsilon}-u_{0}-\varepsilon\chi^{\varepsilon}S_{\varepsilon
}(\nabla\widetilde{u}_{0})-\varepsilon\chi_{0}^{\varepsilon}S_{\varepsilon
}(\widetilde{u}_{0})\right\Vert _{H^{1}(\Omega)}\leq C\varepsilon^{\frac{1}%
{2}}\left\Vert u_{0}\right\Vert _{H^{2}(\Omega)}. \label{28}%
\end{equation}

\end{theorem}

\subsection{Proof of Theorem \ref{t2.1}}

Our aim in this section is to prove estimate (\ref{5.9}). To achieve that, we
proceed like in \cite{Zhikov2} to first show that
\begin{equation}
\left\Vert u_{\varepsilon}\right\Vert _{H^{1}(\Gamma_{2\varepsilon})}\leq
C\varepsilon^{\frac{1}{2}}\left\Vert f\right\Vert _{L^{2}(\Omega)} \label{**}%
\end{equation}
where $\Gamma_{2\varepsilon}=\{x\in\Omega:\mathrm{dist}(x,\partial
\Omega)<2\varepsilon\}$. Indeed, writing $u_{\varepsilon}$ as
\[
u_{\varepsilon}=(u_{\varepsilon}-u_{0}-\varepsilon\chi^{\varepsilon
}S_{\varepsilon}(\nabla\widetilde{u}_{0})-\varepsilon\chi_{0}^{\varepsilon
}S_{\varepsilon}(\widetilde{u}_{0}))+u_{0}+\varepsilon\chi^{\varepsilon
}S_{\varepsilon}(\nabla\widetilde{u}_{0})+\varepsilon\chi_{0}^{\varepsilon
}S_{\varepsilon}(\widetilde{u}_{0}),
\]
we see that
\begin{align}
\left\Vert u_{\varepsilon}\right\Vert _{H^{1}(\Gamma_{2\varepsilon})}  &
\leq\left\Vert u_{\varepsilon}-u_{0}-\varepsilon\chi^{\varepsilon
}S_{\varepsilon}(\nabla\widetilde{u}_{0})-\varepsilon\chi_{0}^{\varepsilon
}S_{\varepsilon}(\widetilde{u}_{0})\right\Vert _{H^{1}(\Gamma_{2\varepsilon}%
)}+\left\Vert u_{0}\right\Vert _{H^{1}(\Gamma_{2\varepsilon})}\label{29}\\
&  +\left\Vert \varepsilon\chi^{\varepsilon}S_{\varepsilon}(\nabla
\widetilde{u}_{0})\right\Vert _{H^{1}(\Gamma_{2\varepsilon})}+\left\Vert
\varepsilon\chi_{0}^{\varepsilon}S_{\varepsilon}(\widetilde{u}_{0})\right\Vert
_{H^{1}(\Gamma_{2\varepsilon})}.\nonumber
\end{align}
First of all, since $\Omega$ is a $\mathcal{C}^{1,1}$ bounded domain in
$\mathbb{R}^{d}$ and the coefficients of the homogenized operator
$\mathcal{P}_{0}=-\operatorname{div}(\widehat{A}\nabla+\widehat{V}%
)+\widehat{B}\cdot\nabla+\widehat{a}_{0}+\mu$ are constant, it follows from
\cite[Theorem 4, pp. 334-335]{Evans} that
\begin{equation}
\left\Vert u_{0}\right\Vert _{H^{2}(\Omega)}\leq C\left\Vert f\right\Vert
_{L^{2}(\Omega)}. \label{eq0}%
\end{equation}
It follows from (\ref{28}) and (\ref{eq0}) that
\[
\left\Vert u_{\varepsilon}-u_{0}-\varepsilon\chi^{\varepsilon}S_{\varepsilon
}(\nabla\widetilde{u}_{0})-\varepsilon\chi_{0}^{\varepsilon}S_{\varepsilon
}(\widetilde{u}_{0})\right\Vert _{H^{1}(\Gamma_{2\varepsilon})}\leq
C\varepsilon^{\frac{1}{2}}\left\Vert f\right\Vert _{L^{2}(\Omega)}.
\]
For the other terms involved in the right-hand side of (\ref{29}), we may
proceed as in \cite[Lemma 3.3]{Zhikov2} to derive (\ref{**}).

This being so, the theorem will be proven once we will check the estimate
\begin{equation}
\left\Vert z_{\varepsilon}\right\Vert _{L^{2}(\Omega)}\leq C\varepsilon
\left\Vert f\right\Vert _{L^{2}(\Omega)}. \label{eq1}%
\end{equation}
Indeed, first note that (\ref{20}) implies
\[
\left\Vert u_{\varepsilon}-u_{0}-\varepsilon\chi^{\varepsilon}S_{\varepsilon
}(\nabla\widetilde{u}_{0})-\varepsilon\chi_{0}^{\varepsilon}S_{\varepsilon
}(\widetilde{u}_{0})+z_{\varepsilon}\right\Vert _{L^{2}(\Omega)}\leq
C\varepsilon\left\Vert f\right\Vert _{L^{2}(\Omega)}.
\]
Therefore it is clear from the series of inequalities
\[
\left\Vert \varepsilon\chi^{\varepsilon}S_{\varepsilon}(\nabla\widetilde
{u}_{0})+\varepsilon\chi_{0}^{\varepsilon}S_{\varepsilon}(\widetilde{u}%
_{0})\right\Vert _{L^{2}(\Omega)}\leq C\varepsilon\left\Vert u_{0}\right\Vert
_{H^{2}(\Omega)}\leq C\varepsilon\left\Vert f\right\Vert _{L^{2}(\Omega)}%
\]
that proving (\ref{eq1}) implies (\ref{5.9}). So, we concentrate on the proof
of (\ref{eq1}).

We follow an argument developed above in the proof of Lemma \ref{l6.5} by
setting $v_{\varepsilon}=z_{\varepsilon}-\phi_{\varepsilon}$ where
$z_{\varepsilon}$ and $\phi_{\varepsilon}$ are defined by (\ref{5.3}) and
(\ref{5.36}) respectively. Then $v_{\varepsilon}\in H_{0}^{1}(\Omega)$ and
$\mathcal{P}_{\varepsilon}v_{\varepsilon}=-\mathcal{P}_{\varepsilon}%
\phi_{\varepsilon}$ in $\Omega$. We know from (\ref{eq01}), (\ref{eq02}) and
(\ref{eq0}) that
\begin{equation}
\left\Vert \nabla\phi_{\varepsilon}\right\Vert _{L^{2}(\Omega)}\leq
C\varepsilon^{\frac{1}{2}}\left\Vert f\right\Vert _{L^{2}(\Omega
)},\ \ \left\Vert \phi_{\varepsilon}\right\Vert _{L^{2}(\Omega)}\leq
C\varepsilon\left\Vert f\right\Vert _{L^{2}(\Omega)}.\label{eq03}%
\end{equation}
Next, let $F\in L^{2}(\Omega)$ and let $t_{\varepsilon}\in H_{0}^{1}(\Omega)$
be the unique solution (for each $\varepsilon>0$) of
\begin{equation}
\mathcal{P}_{\varepsilon}^{\ast}t_{\varepsilon}=F\text{ in }\Omega\label{eq04}%
\end{equation}
where $\mathcal{P}_{\varepsilon}^{\ast}$ is the adjoint of $\mathcal{P}%
_{\varepsilon}$ ($\mathcal{P}_{\varepsilon}^{\ast}$ is obtained from
$\mathcal{P}_{\varepsilon}$ by interchanging $V^{\varepsilon}$ and
$B^{\varepsilon}$: $\mathcal{P}_{\varepsilon}^{\ast}=-\operatorname{div}%
(A^{\varepsilon}\nabla+B^{\varepsilon})+V^{\varepsilon}\nabla+a_{0}%
^{\varepsilon}+\mu$). Since $\mathcal{P}_{\varepsilon}^{\ast}$ has a similar
form as $\mathcal{P}_{\varepsilon}$, we infer the existence of $t_{0}\in
H_{0}^{1}(\Omega)$ such that $t_{\varepsilon}\rightarrow t_{0}$ in $H_{0}%
^{1}(\Omega)$-weak (as $\varepsilon\rightarrow0$), where $t_{0}$ solves the
equation $\mathcal{P}_{0}^{\ast}t_{0}=F$ in $\Omega$. Repeating the
homogenization process as before, we are able to show that (see (\ref{**}))
\begin{equation}
\left\Vert \nabla t_{\varepsilon}\right\Vert _{H^{1}(\Gamma_{2\varepsilon}%
)}\leq C\varepsilon^{\frac{1}{2}}\left\Vert F\right\Vert _{L^{2}(\Omega
)}.\label{eq05}%
\end{equation}
Now, it is a fact that proving (\ref{eq1}) amounts in checking that
\begin{equation}
\left\Vert v_{\varepsilon}\right\Vert _{L^{2}(\Omega)}\leq C\varepsilon
\left\Vert f\right\Vert _{L^{2}(\Omega)}\label{eq06}%
\end{equation}
because of the second inequality in (\ref{eq03}). So, to prove (\ref{eq06}) we
write
\[
\int_{\Omega}Fv_{\varepsilon}dx=\left(  \mathcal{P}_{\varepsilon}^{\ast
}t_{\varepsilon},v_{\varepsilon}\right)  =\left(  t_{\varepsilon}%
,\mathcal{P}_{\varepsilon}v_{\varepsilon}\right)  =-(t_{\varepsilon
},\mathcal{P}_{\varepsilon}\phi_{\varepsilon})=-(\mathcal{P}_{\varepsilon
}^{\ast}t_{\varepsilon},\phi_{\varepsilon})=-\mathcal{B}_{\varepsilon,\Omega
}(\phi_{\varepsilon},t_{\varepsilon})
\]
where $\mathcal{B}_{\varepsilon,\Omega}(\cdot,\cdot)$ is defined by
(\ref{*0}). But \textrm{supp}$\phi_{\varepsilon}\subset$ \textrm{supp}%
$\theta_{2\varepsilon}\subset\Omega_{2\varepsilon}=\{x\in\mathbb{R}%
^{d}:\mathrm{dist}(x,\partial\Omega)<2\varepsilon\}$, so that
\begin{align*}
\int_{\Omega}Fv_{\varepsilon}dx &  =-\mathcal{B}_{\varepsilon,\Gamma
_{2\varepsilon}}(\phi_{\varepsilon},t_{\varepsilon})\\
&  =-\int_{\Gamma_{2\varepsilon}}(A^{\varepsilon}\nabla\phi_{\varepsilon}%
\cdot\nabla t_{\varepsilon}+\phi_{\varepsilon}V^{\varepsilon}\cdot\nabla
t_{\varepsilon}+(B^{\varepsilon}\cdot\nabla\phi_{\varepsilon})t_{\varepsilon
}+(a_{0}^{\varepsilon}+\mu)\phi_{\varepsilon}t_{\varepsilon})dx.
\end{align*}
It follows from the continuity of $\mathcal{B}_{\varepsilon,\Gamma
_{2\varepsilon}}$ that
\begin{equation}
\left\vert \int_{\Omega}Fv_{\varepsilon}dx\right\vert \leq C\left\Vert
\nabla\phi_{\varepsilon}\right\Vert _{L^{2}(\Omega)}\left\Vert \nabla
t_{\varepsilon}\right\Vert _{L^{2}(\Gamma_{2\varepsilon})}\leq C\varepsilon
\left\Vert f\right\Vert _{L^{2}(\Omega)}\left\Vert F\right\Vert _{L^{2}%
(\Omega)}.\label{eq07}%
\end{equation}
Since $F$ is arbitrary in (\ref{eq07}), this leads at once to (\ref{eq06}).
This concludes the proof of Theorem \ref{t2.1}.

\begin{acknowledgement}
\emph{The work of J.L. Woukeng has been carried out under the support of the
Alexander von Humboldt Foundation. He gratefully acknowledge the Foundation. G.C. is a member of GNAMPA (INDAM).}
\end{acknowledgement}

\end{document}